\documentclass[12pt]{amsart}
\usepackage{amsmath,amsfonts,amssymb,comment} 
\usepackage{pdfpages}
\usepackage{amsbsy}
\usepackage{amscd}

\usepackage{amsthm}
\usepackage{bbold}
\usepackage{mathrsfs}
\usepackage{verbatim}
\usepackage [all]{xy}
\usepackage[top=1in,bottom=1in,left=1.25in,right=1.25in]{geometry}

\usepackage{tikz-cd}

\newcommand{\hide}[1]{}

\theoremstyle{plain}
\newtheorem{thm}{Theorem}[section]
\newtheorem{prop}[thm]{Proposition}
\newtheorem{claim}[thm]{Claim}

\newtheorem{cor}[thm]{Corollary}
\newtheorem{lem}[thm]{Lemma}

\newtheorem*{Cartan-Dieudonne}{Cartan-Dieudonne theorem {\rm (\cite [Chapter I, Theorem 7.1]{Lam})}}

\theoremstyle{definition}
\newtheorem{defi}[thm]{Definition}
\newtheorem{rem}[thm]{Remark}
\newtheorem{nota}[thm]{Notation}

\theoremstyle{remark}

\newcommand{\HH}{{\mathbb H}}

\newcommand{\CC}{{\mathbb C}}
\newcommand{\QQ}{{\mathbb Q}}
\newcommand{\RR}{{\mathbb R}}

\newcommand{\PP}{{\mathbb P}}

\title{Twistor triangles in the period domain of complex tori}
\author{Nikolay Buskin}

\address{Department of Mathematics, University of California San Diego, 9500 Gilman Drive \# 0112, La Jolla, CA 92093-0112, USA}

\email{nvbuskin@gmail.com}

\begin{document}
\begin{abstract}
We study the geometry of (generalized) twistor triangles $\triangle J_1J_2J_3$ in the period domain
of compact complex tori of  complex dimension $2n$ by the means of the representation theory 
of algebras (of real dimension 8) generated by the complex structure operators $J_1,J_2,J_3$. 
Considering the period domain as a homogeneous space for $G=GL_{4n}(\RR)$,
we introduce on it a $G$-invariant pseudometric and define
pseudometric invariants, helping us (generally) to distinguish triangles from a reasonably defined class up to $G$-equivalence.
\end{abstract}
\setcounter{tocdepth}{1}
\sloppy
\maketitle
\tableofcontents
\section{Introduction}

We call a manifold $M$ of a real dimension $4m$ a {\it hypercomplex} manifold, if  there exist 
(integrable) complex structures
$I,J,K$ on $M$ satisfying the quaternionic relations $$I^2=J^2=K^2=-Id, IJ=-JI=K.$$
The ordered triple $(I,J,K)$ is called a {\it hypercomplex structure} on $M$.

A Riemannian $4m$-manifold $M$ with a metric $g$ is called {\it hyperk\"ahler} with 
respect to $g$ (see \cite[p. 548]{Hitchin}), 
if there exist complex structures $I$, $J$ and $K$ on $M$, such that 
$I,J,K$ are covariantly constant and 
are isometries of the tangent bundle $TM$ with respect to $g$, satisfying
the above quaternionic relations.
We call the  ordered triple $(I,J,K)$ of such complex structures {\it a hyperk\"ahler
structure on $M$ compatible with $g$}. 

Every hyperk\"ahler manifold $M$ naturally carries the  underlying hypercomplex structure
and is thus hypercomplex.
A hypercomplex 
 structure $(I,J,K)$ gives rise to a sphere $S^2$ of complex structures on $M$:
$$S^2=\{aI+bJ+cK| a^2+b^2+c^2=1\}.$$

We call the family $\mathcal M=\{(M,\lambda)| \lambda \in S^2\}  \rightarrow S^2$ a {\it twistor family over the twistor sphere $S^2$}. The family $\mathcal M$ can be endowed with a complex structure,
so that it becomes a complex manifold and the fiber $\mathcal M_\lambda$ is biholomorphic to
the complex manifold $(M,\lambda)$, see \cite[p. 554]{Hitchin}.

The well known examples of compact hyperk\"ahler manifolds are even-dimensional complex tori and irreducible holomorphic symplectic manifolds ({\it IHS manifolds}).
We recall that an IHS manifold is a simply connected compact K\"ahler manifold $M$  with $H^0(M,{\Omega}_M^2)$ generated by
an everywhere non-degenerate holomorphic 2-form $\sigma$.

It is known that in the period domain of 
an IHS manifold any two periods can be connected
by a path of twistor lines arising from the corresponding hyperk\"ahler structures, see 
the work of Verbitsky, \cite{Verb-Torelli}, and its short exposition in  \cite{Bourbaki}. 
The twistor path connectivity of each of the two connected components of the period domain
of complex tori was proved in \cite{Twistor-lines}. 

Let us recall the construction of this period domain.
Let   $V_\RR$ be a real vector space of real dimension
$4n$.  The compact complex tori of complex dimension $2n$,
as real smooth manifolds, are quotients $V_\RR/\Gamma$ of $V_\RR$ by a lattice $\Gamma$
and the complex structure of such a torus is given by an endomorphism $I\colon V_\RR \rightarrow V_\RR, I^2=-Id$.
Following \cite{Twistor-lines}, we denote the period domain of compact complex tori of complex dimension $2n$ by $Compl$, as a
set of imaginary endomorphisms of $V_\RR$ it is diffeomorphic to the orbit $G\cdot I$, where 
$G=GL(V_\RR)=GL(4n,\RR)$ acts via the adjoint action, $g \cdot I =g(I)=gIg^{-1}$. We have $G\cdot I\cong G/G_I$,
where $G_I$ is the adjoint action stabilizer of $I$, $G_I=GL((V_\RR,I))\cong GL_{2n}(\CC)$. 
As the $Ad\, G$-action 
is the only action we will be dealing with, we will simply refer to it as the $G$-action. 
A twistor sphere $S=S(I,J)\subset Compl \subset End\,V_\RR$ determines an embedding 
of the algebra of quaternions $\HH \hookrightarrow End\,V_\RR$, we call the image of such embedding
{\it the algebra of quaternions associated with $S$}. We define $G_\HH\subset G$ to be the pointwise stabilizer of  $\HH$ in $End\,V_\RR$, or, what is the same, of the sphere $S$.
 We obviously have $G_\HH=G_I\cap G_J$. 

The $G$-action on $Compl$ naturally extends to the $G$-action on subsets of $Compl$,
in particular, on twistor lines and on configurations of those.

 The period domain $Compl$ consists
of two connected components, corresponding to two connected components of
$G$. We have the embedding of $Compl$ into the Grassmanian $Gr(2n,V_\CC)$ of $2n$-dimensional
complex subspaces in $V_\CC=V_\RR\otimes \CC$ given by $$Compl \ni I \mapsto (Id-iI)V_\RR \in Gr(2n,V_\CC),$$ which maps $Compl$ biholomorphically onto an open subset of $Gr(2n,V_\CC)$, whose
complement is the real-analytic locus $\mathcal L_\RR =\{U \in Gr(2n,V_\CC)|\, U\cap V_\RR \neq \{0\}\}$
of $2n$-dimensional complex subspaces in $V_\CC$ having nontrivial intersection with $V_\RR$.
This locus $\mathcal L_\RR$ is of real codimension 1 in $Gr(2n,V_\CC)$ and it cuts $Gr(2n,V_\CC)$ into two
pieces each of which is the corresponding component of $Compl$, the components correspond,
non-canonically, to the connected components $GL^+(V_\RR),GL^-(V_\RR)$ 
of $GL(V_\RR)$. 

For further discussion of twistor lines and the configurations of those we 
need the following lemma, which summarizes technical results proved in  \cite{Twistor-lines}, see also  \cite{Generalized-twistor-lines}.
\begin{lem}
\label{Two-point-lemma}
Let $S_1,S_2\subset Compl \subset End\,V_\RR$ be any two  twistor lines. 
If the intersection $S_1\cap S_2$ contains  points that are linearly independent as vectors in 
$End\,V_\RR$, 
then $S_1=S_2$. In particular, any two distinct twistor lines $S_1,S_2\subset Compl$
are either disjoint or $S_1\cap S_2$ consists of a pair of antipodal points $\pm I$. 
If $S\subset Compl$ is a twistor line and $I_1,I_2\in S$ are linearly independent,
then $G_{I_1}\cap G_{I_2}=G_\HH$, where $\HH\subset End\,V_\RR$ is the algebra of quaternions associated
with $S$. 
\end{lem}

This lemma tells us that every twistor line $S$ is uniquely determined by  any two non-proportional points
$I_1,I_2$ in $S$, allowing us to write $S=S(I_1,I_2)$ (here $I_1,I_2$ need not anticommute). Note,
that it is not true that any two points $I_1,I_2 \in Compl$ belong to a twistor sphere (this will actually be explained later).

Let $I_1,I_2,I_3$ be complex structure operators on  $V_\RR$, belonging to the same twistor sphere $S\subset Compl$  
  and linearly independent as vectors in $End\,V_\RR$. 
We are not assuming here that $I_1,I_2$ and $I_3$ satisfy quaternionic identities.
By Lemma \ref{Two-point-lemma} we have $G_\HH=G_{I_1}\cap G_{I_2}=G_{I_1}\cap G_{I_2}\cap G_{I_3}$. 
The main result of \cite{Twistor-lines} is that the triple intersection of submanifolds
$G_{I_1}/G_{\HH}, G_{I_2}/G_\HH$ and $G_{I_3}/G_\HH$ in $G/G_\HH$
at $eG_\HH$ is transversal (\cite[Prop. 3.5]{Twistor-lines}).

The transversality
at  $eG_\HH$ means that for  every triple $(g_1,g_2,g_3)$ with $g_{j} \in G_{I_j}, j=1,2,3$, close enough to $e \in G$ 
(and thus determining $g_jG_\HH \in G/G_\HH$ 
close enough, in the respective topology, to $eG_\HH$),
we have that $g_{1}g_{2}g_{3} \in G_\HH$ if and only if $g_{j} \in G_\HH$ for every $j=1,2,3$.
Speaking informally, the groups $G_{I_1},G_{I_2},G_{I_3}$  are independent (modulo $G_\HH$) near 
$G_{I_1}\cap G_{I_2}\cap G_{I_3}=G_\HH$. 

One may ask if there exist general (not necessarily close to $e$)
$g_{1} \in G_{I_1},g_{2}\in G_{I_2},g_{3}\in G_{I_3}$,
 such that
we have the relation $g_{1}g_{2}g_{3} \in G_\HH$, and, in general, 
one can ask
what is the whole fiber $m^{-1}(G_\HH)$ of the multiplication map
$$m\colon G_{I_1}\times G_{I_2}\times G_{I_3} \rightarrow G, (g_1,g_2,g_3)\mapsto g_1g_2g_3,$$
where the Cartesian product is merely  a product of sets.
%
Again, informally, this question is about how ``independent'' 
the subgroups $G_{I_j}\subset G$ are in global and what kind of relations of the specified type may arise. 
We answer this question in Theorem \ref{Theorem-relation}, where we give an explicit description of the fiber  $m^{-1}(G_\HH)$. On our way to the formulation of Theorem \ref{Theorem-relation} we need to develop
some geometry of twistor lines related to $m^{-1}(G_\HH)$.  
\subsection{Triangles} 
%
%
Let us consider a more general relation $g_{1}g_{2}g_{3}(S)=S$, that is,
$(g_1,g_2,g_3) \in m^{-1}(G_S)$, where $G_S$ is the stabilizer in $G$ of $S$
as a set. 
Assume there is a triple $(g_{1},g_{2},g_{3}) \in m^{-1}(G_S)$ 
%
 and this triple is sufficiently nontrivial,
in the sense that the twistor lines $S, g_{2}(S)$ and $g_{2}g_{3}(S)=g_{1}^{-1}(S)$ 
are all distinct.
Then these twistor lines $S, g_{2}(S)$ and $g_{2}g_{3}(S)=g_{1}^{-1}(S)$
are consecutive, that is, their pairwise intersections are nonempty, and those are actually the pairs
of points (listed in the respective order) $\{\pm I_2\}=S \cap g_{2}(S), \{\pm g_2(I_3)=\pm g_{2}(g_{3}(I_3))\} =g_{2}(S \cap g_{3}(S))=g_{2}(S)\cap g_{1}^{-1}(S)$ and 
$\{\pm I_1\}=g_{1}^{-1}(S)\cap S$. 
Thus we obtain a triangle, formed by the three consecutive twistor lines.

\vspace*{0.5cm}
\input{triangle1.pic}
\vspace*{-12.5cm}
\begin{center}
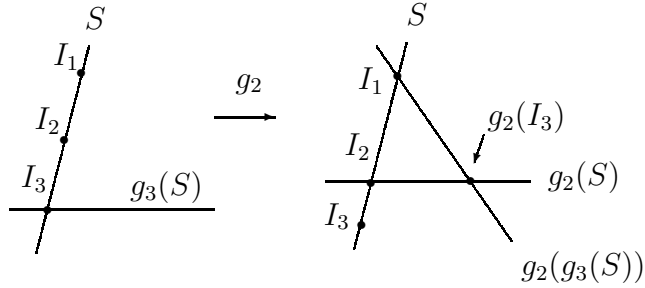
Picture 1: Obtaining a twistor triangle from $g_{1}g_{2}g_{3}(S)=S$.
\end{center}

On the opposite, given three consecutive twistor lines, we can find three complex structures $I_1,I_2,I_3 \in S$,
where $S$ is one of these lines, and elements $g_{j}\in G_{I_j},j=1,2,3$, 
such that $(g_1,g_2,g_3)\in m^{-1}(G_S)$, that is, $S=g_{1}g_{2}g_{3}(S),g_{2}(S),g_{2}(g_{3}(S))$ constitute our triple of consecutive lines. Indeed, let $S_1,S_2,S_3$ be the consecutive twistor lines. 
Choose $I_1\in S_1 \cap S_3$ and $I_2\in S_1 \cap S_2$. As we know from \cite{Twistor-lines}
 or \cite{Generalized-twistor-lines}, 
the $G$-action stabilizer $G_{I}\subset G$ of $I\in Compl$ acts transitively on the set of twistor spheres containing $I$, so that we can find elements $g_2\in G_{I_2}$ such that $S_2=g_2(S_1)$
and $g_1 \in G_{I_1}$ such that $S_3=g_1^{-1}(S_1)$. 
 Next, choose $J_3\in S_2 \cap S_3$ and set $I_3=g_2^{-1}(J_3)\in S$, choose 
$f_3\in G_{J_3}$ such that $f_3(S_2)=S_3=g_1^{-1}(S_1)$. Then setting $g_3=g_2^{-1}f_3g_2$
we get that $g_3\in G_{I_3}$ and $g_1g_2g_3(S_1)=g_1g_2\cdot g_2^{-1}f_3g_2(S_1)=
g_1f_3(S_2)=g_1(g_1^{-1}(S_1))=S_1$, so that $(g_1,g_2,g_3)\in m^{-1}(G_{S_1})$, as required.

%
Further we give a rigorous definition of 
a (generalized) twistor triangle and relate to every twistor triangle
a certain real associative algebra $\mathcal H$ of dimension $8$. 
%
The  properties of the algebras $\mathcal H$ are formulated in  Theorem \ref{Algebraic-theorem}. 
The representation theory of such algebras, summarized in Theorem \ref{Main-theorem}, 
will allow us to prove Theorem \ref{Theorem-relation}. 

%

\subsection{Generalized triangles}It is natural to generalize
the notion of a twistor triangle, in order to proceed with the classification of
representations of the associated algebras $\mathcal H$. Let us explain this
generalization. 

It is easy to see, and this is explained in \cite{Generalized-twistor-lines} that two non-proportional complex structures
$J_1,J_2$ belong to the same, uniquely defined, twistor sphere $S$ if and only if $J_1J_2+J_2J_1=2\alpha Id$
for some $\alpha \in \RR$ such that $|\alpha|<1$ (such $J_1$ and $J_2$ generate the subalgebra  
$\HH\subset End\, V_\RR$ associated with $S$). This fact provides a natural generalization
of the notion of a twistor sphere, namely,  if $J_1J_2+J_2J_1=2\alpha Id$ for some general 
$\alpha \in \RR$ and $J_1\neq \pm J_2$, then there is a canonically defined 
complex-analytic curve $S(J_1,J_2)$
in $Compl$ containing $\pm J_1,\pm J_2$, it is the intersection of the subalgebra in $End\,V_\RR$, generated by $J_1,J_2$ with $Compl\subset End\,V_\RR$.  

In case of $|\alpha|\geqslant 1$ this curve is a non-compact 
curve that we will call {\it a non-compact twistor line}, as opposed to the earlier considered compact twistor lines. If we do not specify whether a twistor line is compact or not, we can talk about it as 
a {\it generalized twistor line}.
The geometry of such curves is studied in \cite{Generalized-twistor-lines}, 
where it is shown, in particular, that the (analytic or Zariski topology) closures of non-compact twistor lines 
in $Gr(2n,V_\CC)\supset Compl$ are $\mathbb{P}^1$'s. 

We generalize accordingly  the notion of a twistor triangle, namely we call 
an ordered triple  of complex structures $(J_1,J_2,J_3)$
{\it a (generalized) twistor triangle $\triangle J_1 J_2 J_3$}, if $J_1J_2+J_2J_1=2\alpha Id$,
$J_2J_3+J_3J_2=2\beta Id$ and 
$J_1J_3+J_3J_1=2\gamma Id$ for some $\alpha, \beta,\gamma \in \RR$ (with no restrictions
on their absolute values now). It is natural  not to require that the sides $S(J_1,J_2), S(J_2,J_3), S(J_3,J_1)$
are all distinct.  Two triangles $\triangle J_1 J_2 J_3$ and $\triangle K_1 K_2 K_3$ are said to be
$G$-{\it equivalent}, if there is $g \in G$ such that $K_l=g(J_l),l=1,2,3$, we emphasize the
importance of the order of vertices. 

Note, that an ordered triple of distinct, pairwise intersecting twistor lines $S_1,S_2,S_3$
does not determine uniquely a twistor triangle, as the intersection
of any two twistor lines $S_i\cap S_j$ consists of two distinct points, so that we indeed need to specify
an ordered triple of points, not only a triple of sides.

Now we introduce the algebra $\mathcal H$, associated to $\triangle J_1 J_2 J_3$,
\begin{multline}
\label{H-definition}
\mathcal H=\mathcal H_{\alpha,\beta,\gamma}=\mathcal H(e_1,e_2,e_3)
=\langle e_1,e_2,e_3\,|\,e_1^2=e_2^2=e_2^2=-1, \\
e_1e_2+e_2e_1=2\alpha, 
e_2e_3+e_3e_2=2\beta,e_3e_1+e_1e_3=2\gamma\rangle.
\end{multline}
This algebra has real dimension $8$. 
By $\mathcal H(J_1,J_2,J_3)$ we denote the homomorphic image of $\mathcal H(e_1,e_2,e_3)$ in $End\, V_\RR$ under the homomorphism $e_i\mapsto J_i, i=1,2,3$. 

The problem of classification of the twistor triangles up to $G$-action is equivalent to the problem of 
 classification of all representations $\rho\colon \mathcal H \rightarrow End\, V_\RR$
up to $G$-isomorphism ($G$-equivalence). 


It is easy to study the irreducible representations of $\mathcal H$ (and thus arbitrary representations) 
when  the 8-dimensional algebra $\mathcal H$ contains the algebra of quaternions $\HH$
as a subalgebra, as then $\dim_\HH \mathcal H=2$ and it is really easy to write down 
the (left or right) regular representations for such $\mathcal H$. 
This is the case, as we have seen, for example, when one of $|\alpha|,|\beta|,|\gamma|$ is strictly less than 1.
In fact, as we will see later, $\mathcal H$ may contain $\HH$ even when none of these strict inequalities holds. 

The above mentioned restricted class of triangles is defined to be the set of those triangles $\triangle J_1 J_2 J_3$ for which the respective algebra $\mathcal H_{\alpha,\beta,\gamma}$ 
(and hence $\mathcal H(J_1,J_2,J_3)$)
contains $\HH$ as a subalgebra, we call such algebra $\mathcal H=\mathcal H_{\alpha,\beta,\gamma}$ {\it quaternionic}. 
%
%
The classification of representations of quaternionic $\mathcal H$, and, thus,
of the triangles from the restricted class, 
 is the content of  Theorem \ref{Main-theorem}.  
Theorem \ref{Main-theorem} relies heavily on Theorem \ref{Algebraic-theorem}, which 
specifies the necessary and sufficient conditions on $\alpha,\beta,\gamma$ in order for  $\mathcal H_{\alpha,\beta,\gamma}$
to contain $\HH$, and proves, in particular,
that, up to isomorphism, there are just three quaternionic algebras $\mathcal H$. 


After all we return to the original question of describing the 
fiber $m^{-1}(G_\HH)$ which is done, as we said  earlier, 
in Theorem \ref{Theorem-relation}.

Let us now get to introducing a machinery, which allows to formulate the ``quaternionic restrictions''
on $(\alpha,\beta,\gamma)$ in a convenient, compact, form.

\subsection{The pseudometric}
\label{subsec-pseudometric}
In this subsection we introduce a pseudometric on $Compl$, which will later be used for defining
the pseudometric invariants of our twistor triangles, that will help us in distinguishing them up to 
$G$-action.  
We define a symmetric bilinear form on $End\, V_\RR$ 
by  $$(A,B)=-\frac{1}{4n}tr(AB).$$
This form is clearly positive on  the vectors corresponding to complex structure operators, that is, vectors
 in $Compl$. 
If we choose  an 
 inner product on $V_\RR$, then there 
is the decomposition $End\,V_\RR =A\oplus S$, where $A$ and $S$ are subspaces of antisymmetric and, respectively, symmetric
operators. The decomposition is orthogonal with respect to the form $(\cdot,\cdot)$, the form $(\cdot,\cdot)$ is positive on $A$, negative on $S$, 
 so that it  has the signature $(8n^2-2n,8n^2+2n)$ (we write the signature of a non-degenerate form as
a pair $(n_+,n_-)$). 
This form is clearly $G$-invariant. 
Let us choose a complex structure operator $I\in Compl$, orthogonal with respect to the inner form on $V_\RR$. 
 Identifying $T_ICompl$ with the subspace of operators, anticommuting with $I$, and further decomposing
$T_ICompl\cong A_I\oplus S_I$ into the respective subspaces of antisymmetric and symmetric operators,
 we can see that $(\cdot,\cdot)|_{T_ICompl}$
has signature $(4n^2-2n, 4n^2+2n)$. As $G$ acts transitively on $Compl$ we see that the restriction of $(\cdot,\cdot)$
to $T_{I_1}Compl$ for every $I_1 \in Compl$ has the same signature, thus $(\cdot,\cdot)|_{TCompl}$ determines a pseudo-riemannian metric
on $Compl$. Note that for a tangent vector $J \in T_ICompl$, $JI=-IJ$ and $J^2=-Id$ we have that 
$(J,J)=-\frac{1}{4n}tr(J^2)>0$,
thus the tangent 2-plane $T_IS$ to an arbitrary compact twistor spheres $S=S(I,J)$, which is explicitly written as 
$T_IS=\langle J,K\rangle$ for $K=IJ$,  is positive  with respect to this pseudo-riemannian metric.

For the case of a generalized twistor line determined by a pair $J_1\neq \pm J_2$ of complex structures,
$J_1J_2+J_2J_1=2\alpha Id$,
we have that the restriction of our indefinite metric to the plane $\langle J_1,J_2\rangle_\RR$
is positive definite if and only if $|\alpha| < 1$, thus, in the latter case
we can define  $\cos \, \sphericalangle J_1J_2 = \frac {(J_1,J_2)}{\sqrt{(J_1,J_1)}\sqrt{(J_2,J_2)}}=-\alpha$.
%
 For $J_1,J_2$ determining a compact twistor sphere 
 the angle $\sphericalangle J_1J_2$ is the length of one of two arcs of the great circle in $S(J_1,J_2)$
through $J_1$ and $J_2$. This is easy to see using the parametrization $t \mapsto e^{tJ}J_1e^{-tJ}$ of the great circle 
in $S$ containing $J_1, J_2$, where
$J \in S$ is a complex structure
 anticommuting with both $J_1,J_2$). 

If $|\alpha|\geqslant 1$ the twistor line determined by $J_1,J_2$ is non-compact, in this case the restriction $(\cdot,\cdot)|_{\langle J_1,J_2\rangle}$
is indefinite, being degenerate precisely when $|\alpha|=1$.

\subsection{The invariants and the formulations of the results}

For a generalized twistor triangle $\triangle J_1J_2J_3$ introduce $$T(\triangle J_1J_2J_3):=
\left(\frac{1}{4n}Tr\, J_1J_2, \frac{1}{4n}Tr\, J_2J_3, \frac{1}{4n}Tr\, J_3J_1\right)=(\alpha,\beta,\gamma).$$
If the triangle $\triangle J_1 J_2 J_3$ is compact, then, as follows from the above discussion, the triple
$T(\triangle J_1 J_2 J_3)$ has a clear geometric meaning, namely 
$T(\triangle J_1 J_2 J_3)=(-\cos\sphericalangle J_1 J_2, -\cos\sphericalangle J_2 J_3, -\cos\sphericalangle J_3 J_1)$. 

Formula (\ref{H-definition}) introduces a real associative algebra 
$\mathcal H$ 
of dimension 8 on three letters $e_1,e_2,e_3$.  
%
In general, a set of generators $f_1,f_2,f_3$ of the algebra $\mathcal H$ that are imaginary units, that is, $f_i^2=-1$, $i=1,2,3$, satisfying the relations $f_1f_2+f_2f_1=2\alpha^\prime,f_2f_3+f_3f_2=2\beta^\prime,f_3f_1+f_1f_3=2\gamma^\prime$
 is called a {\it standard set of generators corresponding to $(\alpha^\prime,\beta^\prime,\gamma^\prime)$} and we say that the triple $(\alpha^\prime,\beta^\prime,\gamma^\prime)$ {\it represents}
$\mathcal H$.  
Note that algebra $\mathcal H$ may be represented by sufficiently different triples, 
so that we may have an isomorphism $\mathcal H=\mathcal H_{\alpha,\beta,\gamma}\cong 
\mathcal H_{\alpha^{\prime},\beta^{\prime},\gamma^{\prime}}$  
for the triple $({\alpha^{\prime},\beta^{\prime},\gamma^{\prime}})$ not reducing to permutations of
the original triple $(\alpha,\beta,\gamma)$ and scalings of the kind $\alpha \mapsto -\alpha$. 

Introduce  a bilinear form on $\mathcal H$, $$q(u,v)=\frac{1}{\dim_\RR\, \mathcal H}Tr(\rho_{reg}(uv)),{\dim_\RR\, \mathcal H}=8,$$
where $\rho_{reg}\colon \mathcal H \to End\,\RR^8$ is the (left or right) regular representation of $\mathcal H$, 
and set $Q_{\alpha,\beta,\gamma}=q|_{\langle e_1,e_2,e_3\rangle}$. The relations of $\mathcal H$ 
easily imply  that the matrix
of $Q_{\alpha,\beta,\gamma}$ in the  basis $e_1,e_2,e_3$ is
\[
Q_{\alpha,\beta,\gamma}=\left(\begin{array}{ccc}
-1 & \alpha & \gamma \\
\alpha & -1 & \beta \\
\gamma & \beta & -1 \\
\end{array}\right).
\]
 
We will also denote such $Q_{\alpha,\beta,\gamma}$  by $Q$. 
We have 
$\det\, Q_{\alpha,\beta,\gamma}=\alpha^2+\beta^2+\gamma^2+2\alpha\beta\gamma-1$. 


The triangle $\triangle \rho_{reg}(e_1)\rho_{reg}(e_2)\rho_{reg}(e_3) \subset  End\, \RR^8,i=1,2,3$  
will be denoted for short by $\triangle e_1 e_2 e_3$.



If the algebra homomorphism $\mathcal H(e_1,e_2,e_3) \rightarrow \mathcal H(J_1,J_2,J_3),e_i\mapsto J_i,$ is 
not an isomorphism, we say that the triangle $\triangle J_1J_2J_3$ is {\it degenerate}.
Introduce the following algebras $\mathcal H(\varepsilon)$ for $\varepsilon=-1,0,1$, 
$$\mathcal H(\varepsilon)=\langle i,j,c|\, i^2=j^2=-1,ij+ji=0, ic=ci,jc=cj,c^2=\varepsilon\rangle\cong \HH\oplus \HH\cdot c$$ with the center $Z(\mathcal H(\varepsilon))=\langle 1,c \rangle$.
The algebra $\mathcal H(-1)$ is classically known as the {\it algebra of biquaternions},
$\mathcal H(1)$ is known as the {\it algebra of split-biquaternions}, and
$\mathcal H(0)$ is  known as the {\it algebra of dual quaternions}.

As we said above, our main result, Theorem \ref{Main-theorem} relies on the
following result, where the signature of a non-degenerate form is written as a pair $(n_+,n_-)$
and the signature of a degenerate form is written as a triple $(n_+,n_-,n_0)$. 
Fix an algebra $\mathcal H$ defined by Formula (\ref{H-definition}). 

\begin{thm}
\label{Algebraic-theorem} Let $(\alpha,\beta,\gamma)$ be any triple of real numbers 
representing $\mathcal H$, $\mathcal H \cong \mathcal H_{\alpha,\beta,\gamma}$.   
The signature of the form $Q=Q_{\alpha,\beta,\gamma}$ 
does not depend on the choice of such $(\alpha,\beta,\gamma)$ and 
is thus an isomorphism invariant of the algebra  $\mathcal H$.
All possible signatures of such forms $Q$ are the nondegenerate cases $(0,3)$, $(2,1)$, $(1,2)$ and the degenerate cases
$(0,2,1)$, $(1,1,1)$, $(0,1,2)$.

The algebra $\mathcal H$ contains a subalgebra of quaternions
$\HH$ precisely when the signature of $Q$ is $(0,3)$, $(0,2,1)$ or $(1,2)$.

The center $Z(\mathcal H)$ has dimension 3 in the case, when $rank\,Q = 1$, that is, 
$Q$ has signature $(0,1,2)$, and this condition determines $\mathcal H$ uniquely, up to isomorphism.
 In all other cases the center has dimension 2.

\medskip
{\bf The case of signature $(0,3)$ (necessarily $\det\, Q< 0$):}
in this case $\mathcal H\cong \mathcal H(1)$. 

 The algebra $\mathcal H$ contains exactly two nontrivial two-sided 
ideals $\HH(1+c), \HH(1-c)$. 

The regular representation $\rho_{reg}$ decomposes as $\rho_1\oplus\rho_2$, where the non-faithful representations $\rho_1,\rho_2$ 
are the only, up to isomorphism, irreducible (non-faithful) 4-representations of $\mathcal H$, $\rho_1,\rho_2$ are given by restrictions
of the regular representations
$\rho_{1}(h)=\rho_{reg}(h)|_{\HH(1+c)},\rho_2(h)=\rho_{reg}(h)|_{\HH(1-c)}$ for all $h \in \mathcal H$, and are isomorphic to the regular representation $\rho_{\HH}$ precomposed, respectively,  with   
 the quotient maps  $\mathcal H\rightarrow \mathcal H/\HH(1-c)\cong\HH$ and $ 
\mathcal H\rightarrow \mathcal H/\HH(1+c)\cong \HH$.

\medskip

{\bf The case of signature $(0,2,1)$
(necessarily $\det\, Q=0$):} in this case $\mathcal H\cong \mathcal H(0)$. 

The algebra $\mathcal H$ contains exactly one nontrivial two-sided 
ideal $\HH c$.  
The 8-representation $\rho_8=\rho_{reg}$ is reducible but not completely reducible. There
are no faithful 4-representations of $\mathcal H$ and there exists exactly one, up to isomorphism, non-faithful
4-representation $\rho_4$, $\rho_4(h)=\rho_{reg}(h)|_{\HH c}, h \in \mathcal H$, isomorphic to the regular
representation $\rho_{\HH}$ precomposed with the quotient map 
$\mathcal H\rightarrow \mathcal H/\HH c \cong \HH$.

\medskip

{\bf The case of signature $(1,2)$ (necessarily $\det\, Q> 0$):}
 in this case $\mathcal H\cong \mathcal H(-1)$. 

 The algebra $\mathcal H$ 
 has no nontrivial two-sided ideals.
All proper left ideals are subspaces of dimension 4 in $\mathcal H$, that are of the form $\HH(w+c), w\in S^2\subset\HH, w^2=-1$.


The regular representation $\rho_{reg}$ decomposes as $\rho_0\oplus \rho_0$, where 
$\rho_0$ is the unique, up to isomorphism, irreducible (faithful) 4-representation of $\mathcal H$,
arising from any proper left ideal $\HH(w+c)$, $\rho_0(h)=\rho_{reg}(h)|_{\HH(w+c)}, h\in \mathcal H$. 
\end{thm}
\begin{rem}
\label{Noncanonical-remark}
 If $e_1,e_2,e_3$ is any standard set of generators corresponding to 
$(\alpha,\beta,\gamma)$, then
the central element $c$ in the formulation of Theorem \ref{Algebraic-theorem}
is proportional to the element $\beta e_1-\gamma e_2+\alpha e_3-e_1e_2e_3$, 
see Proposition \ref{Prop-properties-of-H}. The normalized such $c$, that is,
satisfying $c^2=-1,0$ or $1$ is determined up to a scalar multiple, so the 
definitions of the two-sided ideals
of $\mathcal H\cong \mathcal H(1)$ in Theorem \ref{Algebraic-theorem} and hence the representations $\rho_1,\rho_2$ depend on the (non-canonical!) choice of $c$. 
\end{rem}

Let $$\mathcal D=\{(\alpha,\beta,\gamma) \in \RR^3\,|\, Q_{\alpha,\beta,\gamma}\mbox{ has signature } (0,3),(1,2)
\mbox{ or }(0,2,1)\}.$$

From now we restrict ourselves to  triangles $\triangle J_1J_2J_3$ with
$T(\triangle J_1J_2J_3) \in \mathcal D$, this, by Theorem \ref{Algebraic-theorem},  is the (above mentioned) proper class of twistor triangles,
whose corresponding algebras $\mathcal H_{T(\triangle J_1 J_2 J_3)}$ contain the algebra of quaternions $\HH$. By Theorem \ref{Algebraic-theorem}, up to isomorphism, there are just three such algebras.

The triples of $\alpha,\beta,\gamma$ with $|\alpha|,|\beta|,|\gamma|<1$ corresponding to 
 compact twistor triangles, that is, with compact sides, form a proper subset in $\mathcal D$. 


\begin{defi}
\label{geometric-definition}
The triangle $\triangle J_1J_2J_3$ is called {\it hyperbolic}, if $\det Q_{T(\triangle J_1J_2J_3)}>0$, {\it spherical}, if $\det Q_{T(\triangle J_1J_2J_3)}<0$,
and {\it cylindrical}, if  $\det Q_{T(\triangle J_1J_2J_3)}=0$. The algebra 
$\mathcal H=\mathcal H_{T(\triangle J_1 J_2 J_3)}$ is called {\it hyperbolic} 
($\mathcal H\cong \mathcal H (-1)$), {\it spherical} ($\mathcal H\cong \mathcal H (1)$), or {\it cylindrical} 
($\mathcal H\cong \mathcal H (0)$), 
if the triangle $\triangle J_1J_2J_3$ is such. 
\end{defi}

The part of the latter definition involving the algebra $\mathcal H_{\alpha,\beta,\gamma}$ is correct
because Theorem \ref{Algebraic-theorem} provides that the signature of
the form $Q_{\alpha,\beta,\gamma}$ does not depend on the choice of the 
representing triple $(\alpha,\beta,\gamma)$ 
and the sign of $\det\, Q_{\alpha,\beta,\gamma}$ 
uniquely identifies the signature of $Q_{\alpha,\beta,\gamma}$ when
$(\alpha,\beta,\gamma) \in \mathcal D$. 
\begin{rem}
Let $\triangle J_1 J_2 J_3$ be a nondegenerate generalized  twistor triangle,  with $T(\triangle J_1 J_2 J_3)=(\alpha,\beta,\gamma) \in \mathcal D$. 
For the  algebra $\mathcal H_{\alpha,\beta,\gamma}\cong \mathcal H(J_1,J_2,J_3) \subset End\,V_\RR$ the subset
$\langle J_1,J_2,J_3\rangle \cap Compl=\{aJ_1+bJ_2+cJ_3|(aJ_1+bJ_2+cJ_3)^2=-Id, a,b,c\in \RR\}$ in the cases $\det\,Q_{\alpha,\beta,\gamma}>0, \det\,Q_{\alpha,\beta,\gamma}<0,\det\,Q_{\alpha,\beta,\gamma}=0$, is, respectively, a one-sheeted hyperboloid, a sphere, a cylinder, that contains the 
``geodesic'' segments $\langle J_k,J_l\rangle \cap S(J_k,J_l), 1\leqslant k<l\leqslant 3$, of the 
respective twistor lines, 
joining the vertices $J_1,J_2,J_3$ of our twistor triangle.
 This explains the geometric terminology introduced in  Definition \ref{geometric-definition}.
\end{rem}

\begin{rem}
For a triangle $\triangle J_1J_2J_3$ to be of spherical type means that 
$(\alpha,\beta,\gamma)=T(\triangle J_1 J_2 J_3)$ 
is a triple of minus cosines of lengths of sides of a (geodesic) triangle on a unit 2-sphere.
Note that here we compare triangles using only sides lengths, not saying anything about
comparing their angles. In fact, due to our form $(\cdot,\cdot)$ being indefinite,
it is not always possible to define (in a geometrically meaningful way) the angle 
between the sides of $\triangle J_1J_2J_3$. The exceptional situation, when 
the angles of a spherical twistor triangle are defined and equal, up to taking complements to $\pi$, 
 to the corresponding angles of the 
respective geodesic triangle on a sphere, 
is discussed in Theorem \ref{Theorem-relation}. 
\end{rem}


Let $e_1,e_2,e_3$ be any standard set of generators of $\mathcal H_{\alpha,\beta,\gamma}$
 corresponding to $(\alpha,\beta,\gamma)$ and $c\in \RR(\beta e_1-\gamma e_2+\alpha e_3-e_1e_2e_3) \subset Z(\mathcal H_{\alpha,\beta,\gamma})$ (where, as we know, $c^2\in \RR$) be normalized as in Theorem \ref{Algebraic-theorem}, so that $c^2=-1,0$ or $1$.  We are further using the notations for the irreducible
representations of $\mathcal H_{\alpha,\beta,\gamma}$ introduced in Theorem \ref{Algebraic-theorem}. 


Let us now formulate our result about twistor triangles, in terms of the representation theory of the 
respective algebras $\mathcal H$.

\begin{thm}
\label{Main-theorem}

{\bf Existence.} For every triple $ (\alpha,\beta,\gamma) \in \mathcal{D}$
there exists a (possibly non-faithful) 
representation $\rho\colon \mathcal H=\mathcal H_{\alpha,\beta,\gamma} \rightarrow End\, V_\RR$, $\dim_\RR V_\RR=4n$. 
Moreover, 
if one of the two additional
conditions holds$\colon $ 
 
a) $\mathcal H$ 
is hyperbolic;

or

b) $\mathcal H$ 
 is either spherical or cylindrical 
 and $n>1$;


then there exists a faithful such  $\rho$.
 If $\mathcal H$ 
 is either spherical or cylindrical,
and $n=1$, 
then only a non-faithful $\rho$ exists, 
whose image is a subalgebra of quaternions $\HH \subset End\, V_\RR$. 

 {\bf The number of non-$G$-equivalent representations.}
1) For a hyperbolic algebra $\mathcal H$ there is a unique, up to $G$-equivalence, representation 
$\rho \colon \mathcal H \rightarrow End\, V_\RR$, $\rho=n\rho_0$, which is faithful;

2) For a spherical algebra $\mathcal H$ there are total of $n+1$ classes of $G$-equivalent representations  
$\rho \colon \mathcal H \rightarrow End\,V_\RR, \rho=k\rho_1\oplus (n-k)\rho_2$
(among which there are precisely two non-faithful ones, they correspond to $k=0$ and $k=n$, mapping $\mathcal H$ to $\HH\subset End\,V_\RR$). 
Here $k$ is uniquely identified as $k=\frac{1}{8}(Tr(\rho(c))+4n)$.

3) For a cylindrical algebra $\mathcal H$ there are total of $\lfloor \frac{n}{2}\rfloor +1$ 
non-$G$-equivalent representations $\rho=(n-2k)\rho_4\oplus k\rho_8$ (including the only 
non-faithful one, 
corresponding to $k=0$). Here $k$ is uniquely identified as $k=\frac{1}{4}rk\,\rho(c)$. 
\end{thm}

 The abstract representation theory of algebras $\mathcal H(\varepsilon)$ is elementary
and must be a folklore, nevertheless the nontrivial point of Theorem \ref{Main-theorem} is that the theorem explains this representation theory
with respect to a standard set of generators, that is, with respect to a triangle, from which
our $\mathcal H$ originates.

\begin{defi}
We say that a representation $\rho\colon \mathcal H \rightarrow End\,V_\RR$
of an algebra $\mathcal H$ of spherical type is {\it balanced} if $\rho=k\rho_1\oplus k\rho_2$,
that is, the multiplicities of both $\rho_1$ and $\rho_2$ are equal.
\end{defi}
Theorem \ref{Main-theorem} tells us that a balanced representation of a spherical algebra  
$\mathcal H$ 
exists if and only if $n=\frac{1}{4}dim\,V_\RR$ is an even number.

Let us introduce additional subgroups of the group $G$. Let $S$ be a compact twistor line 
(that is, a 2-sphere), $\HH \subset End\,V_\RR$ be the algebra of quaternions associated to $S$, 
and $I\in S$ be a period.  
 We set $G_{I,S}\subset G_I$ to be the $G_I$-adjoint action stabilizer of $S$ as a set. We note that,
as $G_{I,S}$ is the subgroup of elements of $G_I$, acting as rotations of $S$ about
the ``axis" $\{I,-I\}\subset S$, we have 
 $G_{I,S}\cong 
\langle \exp(tI)h, t\in \RR,h\in G_\HH \rangle \cong SO(2)\times G_\HH
\subset G_S$. 
Then if $I_1,I_2,I_3$ are linearly independent complex structures in $S$, we have
that  $G_S$ is generated by its subgroups $G_{I_j,S},j=1,2,3$, 
and so we have an isomorphism $G_S\cong SO(3)\times G_\HH$.
Fix such $I_1,I_2,I_3\in S$. 
Set $(\alpha,\beta,\gamma)=T(\triangle I_1 I_2 I_3)$. 
 Theorem \ref{Theorem-relation} below states that $\mathcal H_{\alpha,\beta,\gamma}$
is spherical. 

For such $\mathcal H_{\alpha,\beta,\gamma}$ 
Theorem \ref{Main-theorem} allows us to
 choose  representatives $\rho_k\colon \mathcal H_{\alpha,\beta,\gamma} \rightarrow
End\,V_\RR,k=0,\dots,n$, of $n+1$ $G$-equivalence classes of representations of $\mathcal H_{\alpha,\beta,\gamma} \rightarrow
End\,V_\RR$ (not to be confused with the above introduced {\it irreducible} representations!), such that $\rho_k(e_1)=I_1,\rho_k(e_2)=I_2, Tr(\rho_k(c))=4(2k-n), k=0,\dots,n$ 
(again, $\rho_k$ are defined non-canonically, as follows from Remark \ref{Noncanonical-remark}). 
The values $k=0$ and $n$ correspond to the two non-faithful representations, 
$\rho_0(\mathcal H_{\alpha,\beta,\gamma})=\rho_n(\mathcal H_{\alpha,\beta,\gamma})=\HH\subset End\,V_\RR$.  For a representation $\rho\colon \mathcal H_{\alpha,\beta,\gamma} \rightarrow
End\,V_\RR, \rho(e_1)=I_1,\rho(e_2)=I_2,$  we denote by $G_{\HH,\rho}$ the $G_\HH$-action stabilizer of  $\rho$. We set $G_\HH^l$ to be the $l$-fold Cartesian product of $G_\HH$. 
We introduce the following action of $G_\HH^4$ on $m^{-1}(G_\HH)$,
$$(h_1,h_2,h_3,h_4)\cdot (f_1,f_2,f_3)= (h_1f_1h_2^{-1},h_2f_2h_3^{-1},h_3f_3h_4^{-1}).$$ 
Finally we formulate the following answer to our original question about 
the fiber $ m^{-1}(G_\HH)$. 

\begin{thm}
\label{Theorem-relation}
Given three linearly independent complex structures $I_1,I_2,I_3$
in a compact twistor line $S$  with $(\alpha,\beta,\gamma)=T(\triangle I_1 I_2 I_3)$,
the algebra $\mathcal H_{\alpha,\beta,\gamma}$ is spherical. 
The fiber $m^{-1}(G_\HH)$ consists of $n+1$ connected components that are in 
one-to-one correspondence with $G$-equivalence classes of representations $\rho_k\colon
\mathcal H_{\alpha,\beta,\gamma} \rightarrow End\,V_\RR, Tr(\rho_k(c))=4(2k-n),  0\leqslant k \leqslant n$. 

{\bf General components.} 
Each component is a subset of the form $$\{(f_1,f_2,f_3)\in m^{-1}(G_\HH)\,|\, f_2(I_3)\in G_\HH\cdot \rho_k(e_3)\},$$
each such set is an orbit under the action of $G_\HH^4$, 
the orbit is diffeomorphic 
to $G_\HH^4/Stab_{G_\HH^4}(f_1,f_2,f_3)$, where the stabilizer $Stab_{G_\HH^4}(f_1,f_2,f_3)\cong 
G_{\HH,\rho_k}$,
$\dim\, G_\HH^4/Stab_{G_\HH^4}(f_1,f_2,f_3)=12n^2+8nk-8k^2$. 

{\bf The trivial and $SO(3)$-type components.}
For $k=0,n$ we have $G_{\HH,\rho_k}=G_\HH$ 
and the respective orbits are 
 diffeomorphic to $G_\HH^3$. The two orbits are, non-canonically in $k$, the subsets
 $$G_\HH \times G_\HH \times G_\HH \subset 
G_{I_1}\times G_{I_2}\times G_{I_3}$$
(the trivial component)
and 
$$g_1G_\HH \times g_2G_\HH \times g_3G_\HH \subset G_{I_1}\times G_{I_2}\times G_{I_3},$$
where  $g_{j}\in G_{I_j,S}\setminus G_\HH,j=1,2,3$ are unique, up to $G_\HH$, elements satisfying $g_{1}g_{2}g_{3}\in G_\HH$ (the  $SO(3)$-type component).

{\bf The geometry of triangles.} For every $(g_1,g_2,g_3) \in m^{-1}(G_\HH) \cap (G_{I_1}\setminus G_{I_1,S}) \times (G_{I_2}\setminus G_{I_2,S}) \times (G_{I_3}\setminus G_{I_3,S})$  the (compact) twistor triangle $\triangle I_1I_2J_3, J_3=g_{2}(I_3)$, in $Compl$
is nondegenerate and spherical, with $T(\triangle I_1 I_2 J_3)=T(\triangle I_1 I_2 I_3)$, 
so that $\mathcal H_{T(\triangle I_1 I_2 I_3)}=\mathcal H_{T(\triangle I_1 I_2 J_3)} \cong \mathcal H(I_1,I_2,J_3)$. 
If, in addition, the natural representation $\rho\colon \mathcal H(I_1,I_2,J_3) \rightarrow End\,V_\RR$
is balanced, the angles of $\triangle I_1I_2J_3$ are well defined, as 
the angles between the tangent subspaces to the twistor spheres at the vertices,
and they are equal, up to taking complements to $\pi$, to the respective angles of $\triangle I_1I_2I_3$.

\end{thm}

So Theorem \ref{Theorem-relation} tells us that the independence 
of the subgroups $G_{I_1}, G_{I_2}, G_{I_3}$ in terms of the multiplication mapping 
$m\colon G_{I_1}\times G_{I_2}\times G_{I_3}\to G$ in general fails not only at the expected locus
$G_\HH^3$ and at the easy-to-guess locus ``of $SO(3)$-type'' (diffeomorphic to $G_\HH^3$), 
both of which correspond to degenerate triangles, 
but also at (a finite number of) higher dimensional loci in $G_{I_1}\times G_{I_2}\times G_{I_3}$, corresponding
to nondegenerate spherical triangles in $Compl$.

We note here that the problem of the description of the specific fiber
$m^{-1}(G_\HH)$ is  extremely approachable, while it may be difficult, if not impossible at all,
to apply the same methods for describing fibers of the more general type $m^{-1}(g_1g_2g_3G_\HH)$
with $g_j\in G_{I_j}$. 

Now let us sketch the plan of the paper.

In Section \ref{When-section} we prove the part of Theorem \ref{Algebraic-theorem} stating the isomorphism invariance of
the signature  of $Q$, classifying the
possible signatures and specifying the ones that correspond to $\mathcal H$ containing $\HH$.
 Besides that this section contains a summary of algebraic properties of the algebra 
$\mathcal H$, which has a lot of symmetry with respect to a standard set of generators.

In Section \ref{Regular-section} we write down the left regular representation of $\mathcal H$ 
and its irreducible representations, which completes the proof of Theorem \ref{Algebraic-theorem}. 
 Understanding the irreducible representations of $\mathcal H$ allows us to prove
Theorem \ref{Main-theorem}, see Subsection \ref{Main-theorem-subsection}.

Section \ref{Proof of Theorem-relation} 
proves Theorem \ref{Theorem-relation}.

Section \ref{Appendix} contains proofs of some technical statements regarding
the structure of the algebra $\mathcal H$, in particular it describes the center $Z(\mathcal H)$.

\section{When is $\mathcal H_{\alpha,\beta,\gamma}$ a quaternionic algebra?}
\label{When-section}
In this section we prove Theorem \ref{Algebraic-theorem}. 

Let $\mathcal H=\mathcal H_{\alpha,\beta,\gamma}$ be the algebra over $\RR$ given by generators and defining relations $$\langle e_1,e_2,e_3\,|\, 
e^{2}_1=e^2_2=e^2_3=-1, e_1e_2+e_2e_1=2\alpha,
e_2e_3+e_3e_2=2\beta, e_1e_3+e_3e_1=2\gamma\rangle.$$

In the introduction we defined
the bilinear form $q$ on $\mathcal H$, $$q(u,v)=\frac{1}{\dim_\RR \mathcal H}Tr(\rho_{reg}(uv)),
\dim_\RR \mathcal H=8,$$
and, thus, the associated quadratic form $q(v,v)$, which we will also denote by $q$.
Due to the relations of $\mathcal H$, the form $q$ has a lot in common with the vector-valued
quadratic form $Sq$ on $\mathcal H$ that squares the elements of $\mathcal H$, 
$Sq \colon v\mapsto v^2 \in \mathcal H$.


Introduce the subspaces $V=\langle e_1,e_2,e_3\rangle$, $\widetilde{V}=\langle  \beta-e_2e_3, \gamma-e_3e_1, \alpha-e_1e_2 \rangle$ and $c=\beta e_1-\gamma e_2+\alpha e_3-e_1e_2e_3$. 
Set $Q_{\alpha,\beta,\gamma}=q|_V$ and $\widetilde{Q}_{\alpha,\beta,\gamma}=q|_{\widetilde{V}}$ and identify these restrictions with their matrices
in the specified bases of the respective subspaces.

We summarize the properties of algebra $\mathcal H$, in particular, the relation between the form $q$
and the square form $Sq$ on $\mathcal H$ 
 in the following proposition.

\begin{prop}
\label{Prop-properties-of-H}
\medskip

\noindent 1) $Sq|_{V}=q|_{V}$ and $\large Sq|_{\widetilde{V}}=q|_{\widetilde{V}}$;

\medskip

\noindent 2) The matrix $\widetilde{Q}_{\alpha,\beta,\gamma}$ is minus the adjugate of the matrix $Q_{\alpha,\beta,\gamma}$;

\medskip

\noindent 3) We have the $q$-orthogonal decomposition 
$\mathcal H=\RR\cdot 1  \oplus V \oplus \widetilde{V}\oplus \RR\cdot c$;

\medskip

\noindent  4) The element $c=\beta e_1-\gamma e_2+\alpha e_3-e_1e_2e_3$
belongs to  the center of $\mathcal H$.

If $|\alpha|=|\beta|=|\gamma|=1$ and $\gamma=-\alpha\beta$ or, what is the same,  $rk\, Q_{\alpha,\beta,\gamma}=1$ $\iff$ the signature of $Q_{\alpha,\beta,\gamma}$ is $(0,1,2)$, 
the center is 3-dimensional, 
$Z(\mathcal H)=\langle 1,c,z\rangle$ where $z= -\gamma(\alpha-e_1e_2)+(\beta-e_2e_3)-\alpha(\gamma-e_3e_1)$.

If  the signature of $Q_{\alpha,\beta,\gamma}$ is different from $(0,1,2)$,  
 then  the center is 2-dimensional, $Z(\mathcal H)=\langle 1,c \rangle$; 

\medskip

\noindent 5)  $c^2=
-\det Q_{\alpha,\beta,\gamma} \in \RR\cdot 1 \hookrightarrow \mathcal H$;

\medskip
\noindent  6) We have inclusions between subspaces $cV\subset \widetilde{V}$, $c\widetilde{V}\subset V$.
If $\det\,Q_{\alpha,\beta,\gamma}\neq 0$, then these inclusions become equalities and $c$ acts as an involution permuting
these subspaces, and hence also the larger subspaces $\langle 1,V \rangle$, $\langle c, \widetilde{V} \rangle$;

\medskip

\noindent 7) The pairs of elements $\{\alpha-e_1e_2,\gamma-e_3e_1\}, \{\alpha-e_1e_2,\beta-e_2e_3\},
\{\beta-e_2e_3,\gamma-e_3e_1\}$ anticommute, respectively, with $e_1$, $e_2$ and $e_3$;

\medskip
\noindent  8) The form $Sq|_{V\oplus \widetilde{V}}$ has as its range the subspace 
$\langle 1,c \rangle \subset  Z(\mathcal H)$.
%
\end{prop}
The equality of the restrictions of quadratic forms $q$ and $Sq$, stated in part 1, assumes that we identify $\RR$
with $\RR\cdot 1 \subset \mathcal H$. 
The first half of part 1 easily follows from the relations of algebra $\mathcal H$. 
The second half is proved in Section \ref{Appendix}.

In part 3 we easily have $1\perp V$, as $e_i$ are imaginary units, hence $Tr(\rho_{reg}(e_i\cdot 1)=0$
and  we also easily have $1\perp \widetilde{V}$.
The directly verifiable part 7 together with the orthogonality relation $c\perp 1$ (which is equivalent to
$Tr(\rho_{reg}(c))=Tr(\rho_{reg}(e_1e_2e_3))=0$, shown also in Section \ref{Appendix})
and the explained relations $1\perp V, 1 \perp \widetilde{V}$
imply most of the orthogonality relations in part 3,  the remaining ones are verified in Section \ref{Appendix}.

For the proofs of parts 4, 5, 6 and 8 we refer to Section \ref{Appendix}.

The property 2) can be directly verified using the calculations for  1)  in Section \ref{Appendix}
 and writing down the matrices 
of the forms $Q_{\alpha,\beta,\gamma}$ and $\widetilde{Q}_{\alpha,\beta,\gamma}$,
\[
Q_{\alpha,\beta,\gamma}=\left(\begin{array}{ccc}
-1 & \alpha & \gamma \\
\alpha & -1 & \beta \\
\gamma & \beta & -1 \\
\end{array}\right),
\widetilde{Q}_{\alpha,\beta,\gamma}=
-\left(\begin{array}{ccc}
1-\beta^2 & \beta\gamma+\alpha & \alpha\beta+\gamma \\
\beta\gamma+\alpha & 1-\gamma^2 & \alpha\gamma+\beta \\
\alpha\beta+\gamma & \alpha\gamma+\beta  & 1-\alpha^2
\end{array}\right).
\]
we see that $\widetilde{Q}_{\alpha,\beta,\gamma}$ is minus the adjugate matrix of $Q_{\alpha,\beta,\gamma}$. We will further use the shorter notations  $Q=Q_{\alpha,\beta,\gamma}$ and $\widetilde{Q}=\widetilde{Q}_{\alpha,\beta,\gamma}$.


\begin{thm}
\label{Signature-theorem}
 If the form $q$ is non-degenerate, then its signature is one of the three
$(2,6), (6,2)$ or $(4,4)$. 
In the cases of signature $(2,6)$ and $(4,4)$ the algebra $\mathcal H$
contains $\HH$ as a subalgebra. In the case of signature $(6,2)$ $\mathcal H$
does not contain $\HH$.
If $q$ is degenerate, its signature is $(1,3,4)$, $(3,1,4)$ or $(1,1,6)$.
For  a  degenerate $q$ the algebra $\mathcal H$ contains $\HH$ only
in the case of signature $(1,3,4)$.
\end{thm}

Note that the part of the statement about $\mathcal H$ not containing $\HH$
in the case of $q$ of signatures $(6,2)$, $(3,1,4)$ and $(1,1,6)$ is trivial: indeed, if there is $\HH \subset \mathcal H$ 
then the restriction $q|_{\HH}$  must have signature $(1,3)$, which is not possible
in the specified cases.

\begin{proof}[Proof of Theorem \ref{Signature-theorem}]
From the definition of $Q$ and the orthogonality relation $1 \perp V$, which is contained in part 3 of 
Proposition \ref{Prop-properties-of-H}, we  have that
the matrix of the restriction of $q$ to the 4-subspace $\langle 1,V\rangle = \langle 1, e_1,e_2,e_3\rangle$
in the basis $1, e_1,e_2,e_3$ is
\[
\left(\begin{array}{cc}
1 & \mathbb{0}_{1\times 3}\\
\mathbb{0}_{3\times 1} & Q \\
\end{array}\right). 
\]
 The definition of $\widetilde{Q}$, the orthogonality relation $c\perp \widetilde{V}$ contained in part 3 of Proposition \ref{Prop-properties-of-H} and part 5 of this proposition  allow us to write down the matrix of the restriction of $q$ 
to the subspace $\langle \widetilde{V},c\rangle$ in the respective basis, 
\[
\left(\begin{array}{cc}
\widetilde{Q} & \mathbb{0}_{3\times 1}\\
\mathbb{0}_{1\times 3} & -\det Q
\end{array}\right).
\] 
As $\widetilde{Q}$  is minus the adjugate matrix of  $Q$, 
the signature of $q$ is completely determined by  the signature of $Q$. 
When $Q$ is non-degenerate, we have the relation $$\widetilde{Q}=-(\det Q)Q^{-1},$$
which shows that $q$ is non-degenerate as well and makes it easy to determine the signature of $\widetilde{Q}$ and of $q$. 

\medskip

Let us now get to classifying all possible signatures of $Q$ and of $\widetilde{Q}$. 
Introducing 
\[
T=\left(\begin{array}{ccc}
1 & \alpha & \gamma \\
0 & 1 & 0 \\
0 & 0 & 1
\end{array}\right),
\mbox{ we get }
T^tQT=\left(\begin{array}{ccc}
-1 & 0 & 0 \\
0 & \alpha^2-1 & \alpha\gamma+\beta \\
0 &  \alpha\gamma+\beta & \gamma^2-1
\end{array}\right).
\]
Introduce the matrix
\[A=
\left(\begin{array}{cc}
\alpha^2-1 & \alpha\gamma+\beta \\
\alpha\gamma+\beta & \gamma^2-1
\end{array}\right),\det\,A=-\det\,Q.
\]
\noindent {\it Case 1.} $A$ has signature $(++)$ if and only if $|\alpha|>1$ and $\det Q<0$ (and
then, as the signature of $Q$ is determined by that of $A$, by choosing appropriate $T$'s we can
see that $|\beta|, |\gamma|>1$ as well). 
Then the signature of $Q$ is $(++-)=(2,1)$ and the signature of $q$ is $(6,2)$.
\medskip

\noindent {\it Case 2.} $A$ has signature $(--)$ if and only if $|\alpha|<1$ and $\det Q <0$ (again, then automatically  
$|\beta|, |\gamma|<1$). The signature of $Q$ in this case is $(---)=(0,3)$ and the signature of $q$ is $(2,6)$.

\medskip

\noindent {\it Case 3.} $A$ has signature $(+-)$ if and only if  $\det Q>0$.
In this case the signature of $Q$ is $(+--)=(1,2)$ and the signature of $q$ is $(4,4)$.
\medskip

In Case 1, as we discussed above, the signature $(6,2)$ guarantees that $\mathcal H$
does not contain $\HH$.

In Case 2 a subalgebra $\HH$ in $\mathcal H$ arises from a pair of anticommuting imaginary units
that can be taken already in the subspace $\langle e_1,e_2,e_3\rangle$. Indeed, if, 
for example, $|\alpha|<1$, then $e_1$ and $e_2$ generate a subalgebra in $\mathcal H$,  
isomorphic to $\HH$, 
as the imaginary unit $\frac{1}{\sqrt{1-\alpha^2}}(\alpha e_1+e_2)$ anticommutes with
$e_1$.

In Case 3 we consider, for example, the plane $P=\langle \alpha-e_1e_2, \gamma-e_3e_1\rangle
\subset \langle e_1,e_2,e_3 \rangle^{\perp}$ and note that actually both of
$\alpha-e_1e_2$ and $\gamma-e_3e_1$ anticommute with $e_1$ (and, of course, one could 
similarly choose analogous anticommuting planes for $e_2$ and $e_3$ as well).
Next, we want to show that $P$ contains an imaginary unit, which,  together with $e_1$,  would give 
us a quaternionic subalgebra $\HH \hookrightarrow \mathcal H$. 
For that we consider the square of a general element of $P$, 
$(x(\alpha -e_1e_2) +y(\gamma-e_3e_1))^2=x^2(\alpha^2-1)-2xy(\alpha\gamma+\beta)
+y^2(\gamma^2-1) \in \RR\cdot 1 \hookrightarrow \mathcal H$.
This is precisely the value $q(x(\alpha -e_1e_2) +y(\gamma-e_3e_1))$, which can be verified
directly or follows by part 1 of Proposition \ref{Prop-properties-of-H}. 
The matrix of $q|_P$
\[\left(\begin{array}{cc}
\alpha^2-1 & -(\alpha\gamma+\beta)\\
-(\alpha\gamma+\beta) & \gamma^2-1
\end{array}\right)
\]
has the determinant equal to $\det A=-\det Q<0$.
So the form $q|_P$ has signature $(+-)$ and it is possible to find a $q$-negative vector
$v=x(\alpha -e_1e_2) +y(\gamma-e_3e_1) \in P, x,y \in \RR$, 
such that $v^2=q(v)=-1 \in \RR\cdot 1 \hookrightarrow \mathcal H$.
  Then the anticommuting pair
$\langle e_1, v \rangle$ determines an embedding 
$\HH \hookrightarrow \mathcal H$. 

\bigskip

If $q$ is degenerate then, as above, we need to consider several cases for $A$.

\medskip

{\it Case 4. $A$ has signature  $(+0)$}. Then $|\alpha|,|\gamma|\geqslant 1$
and they cannot be both equal 1, so that there is at least one of them strictly greater than 1.
Assume, say $|\alpha|>1$ (the subcase when  $|\gamma|>1$ is ruled out in a similar way). We need to consider now $T^t\widetilde{Q}T$,
\[T=
\left(\begin{array}{ccc}
1 & 0 & 0 \\
0 & 1 & 0\\
\frac{\alpha\beta +\gamma}{\alpha^2-1} & \frac{\alpha\gamma +\beta}{\alpha^2-1} & 1\\
\end{array}\right),
T^t\widetilde{Q}T=\frac{1}{\alpha^2-1}
\left(\begin{array}{ccc}
-\det Q & -\alpha \det Q & 0\\
-\alpha \det Q & -\det Q & 0\\
0 & 0 & (\alpha^2-1)^2
\end{array}\right),
\]
which, given that $\det Q=0$ amounts to the signature $(+00)=(1,0,2)$ of $\widetilde{Q}$, which, together with the signature 
$(+-0)=(1,1,1)$ of $Q$ and the signature $(+0)=(1,0,1)$ of $q|_{\langle 1,c\rangle}$ gives the signature of $q$ being $(3,1,4)$. 

\medskip

{\it Case 5. $A$ has signature  $(-0)$}. Then $|\alpha|, 
 |\gamma|\leqslant 1$
and if we have both equalities, then the condition $det\, A=0$ means that $\alpha\gamma+\beta=0$,
so that $A=0$, which is impossible in the current case. So in this case at least one of the absolute values
$|\alpha|,|\gamma|$ is strictly less than 1. 
 We just repeat the arguments above and get that
the signature of $T^t\widetilde{Q}T$ is $(0,1,2)$, which, together with 
the signature $(0,2,1)$ of $Q$ and the signature $(1,0,1)$ of $q|_{\langle 1,c\rangle}$, gives the signature of $q$
being $(1,3,4)$. Note that in this case the fact, that some of $|\alpha|, |\gamma|$
must be strictly less than 1 guarantees that $\HH \hookrightarrow \mathcal  H$.

{\it Case 6. $A=0$}. In this case $|\alpha|=|\gamma|=1$ and $\beta=-\alpha\gamma$ (so that 
$|\beta|=1$ as well), and so we have that
$\alpha\beta+\gamma=\alpha\gamma+\beta=\beta\gamma+\alpha=0$. This means that $\widetilde{Q}$
is the zero matrix and in this case the signature of $Q$ is $(0,1,2)$ and  the signature of $q$  is $(1,1,6)$.

We have seen that if $q$ is degenerate then only in the case of signature $(1,3,4)$ we 
actually get that $\HH$ embeds into $\mathcal H$ and so the proof is now complete.
\end{proof}
\begin{rem}
\label{Universality-remark}
Note that we could use the argument, establishing the embedding $\HH \hookrightarrow \mathcal H$,  
in Case 3 for Case 2 as well,
because $\det A=-\det Q>0$ and the condition $|\alpha|<1$ that we have in Case 2 gives us that $q|_P$ is a negatively definite form.
It was illustrative, however, to emphasize that in Case 2 the embedding $\HH \hookrightarrow \mathcal H$
can be provided by the means of finding an anticommuting pair among the basis elements 
of $V=\langle e_1,e_2,e_3\rangle$ already, without
referring to its orthogonal complement.
\end{rem}

In the course of the proof of Theorem \ref{Signature-theorem}
we have seen that all possible signatures of the restriction $Q=q|_{\langle e_1,e_2,e_3\rangle}$,
determined by the choice of a standard generating set $e_1,e_2,e_3$ for our $\mathcal H$  are in one-to-one correspondence with the signatures of our (independent of choice of generators) form $q$: 
$(2,1)\leftrightarrow (6,2)$, $(0,3)\leftrightarrow (2,6)$, $(1,2)\leftrightarrow (4,4)$,
$(1,1,1)\leftrightarrow (3,1,4)$, $(0,2,1)\leftrightarrow (1,3,4)$, $(0,1,2)\leftrightarrow (1,1,6)$.

That is, indeed the signature of $Q$ does not depend on the choice of a standard generating set and so is an isomorphism invariant of $\mathcal H$. This completes the 
proof of the part of the statement of Theorem \ref{Algebraic-theorem} regarding the signature of $Q$.

\begin{cor}
If $Q$ has any of signatures $(0,3)$, $(0,2,1)$, $(1,2)$, or, what is the same,
$\mathcal H$ contains $\HH$, then  $\mathcal H\cong \mathcal H(\varepsilon)
=\HH\oplus \HH c$ for the central element $c$, proportional to 
$\beta e_1-\gamma e_2+\alpha e_3-e_1e_2e_3$, $c^2=\varepsilon$, $\varepsilon=1,0,-1$
respectively.
\end{cor}

In order to verify the statement of the corollary one just needs to observe 
the (trivial) fact that the (nonzero) central element does not belong to $\HH\subset \mathcal H$,
so, given that $\dim_\RR\mathcal{H}=8$, we have that $\mathcal H=\HH\oplus \HH c$.
 This 
 completes the proof of the part of the statement of Theorem \ref{Algebraic-theorem},
regarding the isomorphism classes of $\mathcal H \supset \HH$. 

\begin{rem}
When $Q=Q_{\alpha,\beta,\gamma}$ is of signature $(0,3)$,  the triangle 
$\triangle e_1 e_2 e_3$
is formed by 2-spheres (which can be considered as twistor spheres in $\mathcal H$
spanned by anticommuting elements,
or, under the regular representation, as twistor spheres in $End\, \RR^8$). 
Then, as we know, $\alpha=-\cos\,\sphericalangle e_1 e_2, 
\beta=-\cos\,\sphericalangle e_2 e_3,\gamma=-\cos\,\sphericalangle e_3 e_1$.
Normalizing the respective basis of $\widetilde{V}$ as 
$f_1=\frac{\beta-e_2e_3}{\sqrt{1-\beta^2}}, f_2=\frac{\gamma-e_3e_1}{\sqrt{1-\gamma^2}},
f_3=\frac{\alpha-e_1e_2}{\sqrt{1-\alpha^2}}$ so as to have $f_1^2=f_2^2=f_3^2=-1$ and observing that 
$\cos\, \sphericalangle f_1 f_2=-q(f_1,f_2)=-\frac{1}{2}(f_1f_2+f_2f_1)=\frac{\alpha+\beta\gamma}{\sqrt{1-\beta^2}\sqrt{1-\gamma^2}}=
-\frac{(-\alpha)-(-\beta)(-\gamma)}{\sqrt{1-\beta^2}\sqrt{1-\gamma^2}}=
-\frac{\cos\,\sphericalangle e_1e_2-\cos\,\sphericalangle e_2e_3 \cdot
\cos\,\sphericalangle e_3 e_1}{\sin\,\sphericalangle e_2 e_3\cdot
\sin\,\sphericalangle e_3 e_1}$, which is equal, by the spherical cosine law, 
to $\cos\, (\pi- \angle e_1e_3e_2)$,
and similarly for the pairs $f_1,f_3$ and $f_2,f_3$, so that the triangle  $\triangle f_1 f_2 f_3$
is also compact and is {\it polar} with respect to $\triangle e_1 e_2 e_3$,
that is, the distances between its vertices are equal $\pi - \angle e_1e_3e_2, 
\pi - \angle e_1e_2e_3, \pi - \angle e_3e_1e_2$, here 
the angles are taken between the geodesic segments lying on the corresponding 2-spheres 
forming the sides of
$\triangle e_1 e_2 e_3$,
and the distances between the vertices of $\triangle f_1f_2f_3$ are measured in the corresponding spheres, forming the sides of $\triangle f_1f_2f_3$. 
Here we extend the classical terminology for triangles  on a unit 2-sphere 
(see, for example, 
\cite[p. 49]{Rat}) to our twistor triangles.
\end{rem}

\section{The representation theory of $\mathcal H$}
\label{Regular-section}
This section is devoted to 
completing the proof of Theorem \ref{Algebraic-theorem} by classifying the irreducible representations of 
$\mathcal H$ and 
 proving Theorem \ref{Main-theorem}. 

Let us reproduce the respective part of the statement of Theorem \ref{Algebraic-theorem}
as a separate proposition.

\begin{prop}
\label{prop-regular}
The algebra $\mathcal H(1)=\HH\oplus \HH c, c^2=1,$ contains exactly two nontrivial two-sided 
ideals $\HH(1+c), \HH(1-c)$. 
The regular representation $\rho_{reg}$ decomposes as $\rho_1\oplus\rho_2$, where the non-faithful representations $\rho_1,\rho_2$ 
are the only irreducible (non-faithful) 4-representations of $\mathcal H$, $\rho_1,\rho_2$ are given by restrictions
of the regular representations
$\rho_{1}(h)=\rho_{reg}(h)|_{\HH(1+c)},\rho_2(h)=\rho_{reg}(h)|_{\HH(1-c)}$ for all $h \in \mathcal H$, and are isomorphic,
respectively, to the regular representation $\rho_{\HH}$ precomposed with   
 the quotient maps  $\mathcal H\rightarrow \mathcal H/\HH(1-c)\cong\HH, 
\mathcal H\rightarrow \mathcal H/\HH(1+c)\cong \HH$.

\medskip

The algebra $\mathcal H(0)=\HH\oplus \HH c, c^2=0$, contains exactly one nontrivial two-sided 
ideal $\HH c$.  
The 8-representation $\rho_8=\rho_{reg}$ is reducible but not completely reducible. There
are no faithful 4-representations of $\mathcal H$ and there exists exactly one non-faithful
4-representation $\rho_4(h)=\rho_{reg}(h)|_{\HH c}$, isomorphic to the regular
representation $\rho_{\HH}$ precomposed with the quotient map 
$\mathcal H\rightarrow \mathcal H/\HH c \cong \HH$.

\medskip

 The algebra $\mathcal H(-1)=\HH\oplus \HH\cdot c, c^2=-1$ 
 has no nontrivial two-sided ideals.
All proper left ideals are subspaces of dimension 4 in $\mathcal H$, that are of the form $\HH(w+c), w\in S^2\subset\HH, w^2=-1$.
%
The regular representation $\rho_{reg}$ decomposes as $\rho_0\oplus \rho_0$, where 
$\rho_0$ is the unique, up to isomorphism, irreducible (faithful) 4-representation of $\mathcal H$,
arising from any proper left ideal $\HH(w+c)$, $\rho_0(h)=\rho_{reg}(h)|_{\HH(w+c)}, h\in \mathcal H$.
\end{prop}

The proof of Proposition \ref{prop-regular} is given in subsections \ref{Spherical-irreps}, \ref{Hyperbolic-irreps}, 
\ref{Cylindrical-irreps}.

\subsection{The case of spherical $\mathcal H=\mathcal H (1)$.}
\label{Spherical-irreps}
A proper  left ideal in $$\mathcal H (1)=\HH \oplus \HH\cdot c, c^2=1,$$
is an $\HH$-submodule of real dimension 4. Denote the generator of such an ideal by $w+c$
for $w\in\HH$.  Then the fact that $\mathcal H (w+c)=\HH(w+c)$
means that  $c(w+c)=q(w+c)$ for some $q \in \HH$. Then $c(w+c)=1+wc=q(w+c)=qw+qc$
means that $q=w$, $qw=w^2=1$, which means $w=\pm 1$ and so we have 
exactly two proper left ideals in $\mathcal H$: $\HH(1+c)\cong \RR^4$ and $\HH(1-c)\cong \RR^4$, with zero intersection, each of which is a two-sided ideal in $\mathcal H$. Clearly $c$ acts on the generator $1+c$ as the identity, so that 
$\rho_{reg}(c)|_{\HH(1+c)}=Id_{\RR^4}$, similarly,  $\rho_{reg}(c)|_{\HH(1-c)}=-Id_{\RR^4}$.

The rest of the statements about the irreducible 4-representations now follows easily.

\subsection{The case of hyperbolic $\mathcal H=\mathcal H (-1)$.}
\label{Hyperbolic-irreps}
A proper  left ideal in $$\mathcal H(-1)=\HH \oplus \HH\cdot c, c^2=-1,$$
is an $\HH$-submodule of real dimension 4. Denote the generator of such an ideal by $w+c$
for $w\in\HH$.  As above, the fact that $\mathcal H (w+c)=\HH(w+c)$
means that  $c(w+c)=q(w+c)$ for some $q \in \HH$. Then $c(w+c)=-1+wc=q(w+c)=qw+qc$
means that $q=w$, $qw=w^2=-1$, which means $w=xi+yj+zk\in S^2\subset \HH, x^2+y^2+z^2=1$. 
 So we have a sphere $S^2$ of (distinct) left ideals $\HH(w+c),w\in S^2$. 
All representations $h\mapsto \rho_{reg}(h)|_{\HH(w+c)}$ are equivalent,
as the right action of the group of unit quaternions $S^3\subset \HH$ on the set of our
left ideals is isomorphic to the conjugation action of $S^3$ on the sphere of imaginary quaternions $S^2$: 
$\HH(w+c)h=\HH(wh+ch)=\HH h(h^{-1}wh+c)=\HH(h^{-1}wh+c)$,
$w \mapsto h^{-1}wh, h\in S^3, w\in S^2$, and the latter action is transitive. 



As the above calculation shows, $c$ acts on $v:=w+c$, $w=xi+yj+zk$, on the left by 
the left multiplication by $q=w=xi+yj+zk$.
The $c$-invariant subspace $\HH v \subset \mathcal H$ is spanned over $\RR$ by vectors $v,iv,jv,kv$.
We have $cv=qv$, $$c\cdot iv=icv=iqv=(-x\cdot 1-zj+yk)v,$$ 
$$c\cdot jv=jcv=jqv=(-y\cdot 1 + zi - xk )v,$$
and $$c\cdot kv=kcv=kqv=(-z\cdot 1 - yi + xj)v.$$
In the basis $v,iv,jv,kv$ the operator of the left multiplication by $c$ has the following matrix
\[
\left(
\begin{array}{cc|cc}
0  &  -x &           -y         & -z \\
x  &   0           & z           & -y \\
\hline
y  &  -z            & 0           & x \\
z   &  y            & -x           & 0 
\end{array}\right).
\]
Set $\rho_0(h)=\rho_{reg}(h)|_{\HH(w+c)}$ for $h \in \mathcal H$. Then,
as $\mathcal H=\HH\oplus \HH c$ and $\rho_0(c),\rho_0(1),\rho_0(i),\rho_0(j),\rho_0(k)$
are linearly independent over $\RR$, we see that $\rho_0$ is faithful.

\subsection{The case of $\mathcal H (0)$.} 
\label{Cylindrical-irreps}
Arguing similarly to the above it is easy to see that $\mathcal H c=\HH c$ is the only proper left ideal in 
$\mathcal H$ 
 (which is also a two-sided ideal). 
In this case we do not have faithful 4-dimensional representations,
the only 4-dimensional representation comes from 
the regular representation  $\rho_\HH$ of $\mathcal H/\mathcal H c \cong \HH$ precomposed with
 the quotient map $\mathcal H\rightarrow \mathcal H/\mathcal H c$. 
 For the  operator $L_c$, acting on $\mathcal H=\HH\oplus \HH c, c^2=0,$ by the left multiplication by $c$, we have the equality
$Ker\,L_c= Im\, L_c$ of its kernel and image. 
 For example, in the $\RR$-base
$c,ic,jc,kc, 1, i,j,k$ our operator $L_c$ has the matrix 
$$\left(
\begin{array}{cc}
\mathbb 0_{4\times 4}  &  \mathbb 1_{4\times 4} \\
\mathbb 0_{4\times 4}  &  \mathbb 0_{4\times 4} \\
\end{array}\right),$$
where 
$\mathbb 1_{4\times 4}$ is the $4\times 4$-identity matrix and $\mathbb 0_{4\times 4}$
is the $4\times 4$-zero matrix.

So, summarizing our observations for this case we conclude that there is exactly one, up to isomorphism, irreducible representation for each of the dimensions 4 and 8, which we call $\rho_4$ and $\rho_8=\rho_{reg}$, 
and only the latter is faithful. 

\bigskip

Thus, the proof of Proposition \ref{prop-regular} is now complete and so is the proof of Theorem \ref{Algebraic-theorem}.

\subsection {General representations of the cylindrical $\mathcal H=\mathcal H(0)$.} For proving Theorem
\ref{Main-theorem} below, we need to show that a general representation of 
$\mathcal H\cong \mathcal H(0)$ is a sum of irreducible representations that arise from its regular representation.
\begin{prop}
\label{Prop-special-representation}
For every representation $\rho \colon \mathcal H(0) \rightarrow End\,V_\RR$ 
 we have $\rho=k\rho_8\oplus l\rho_4$
for appropriate integers $k,l\geqslant 0$. 
\end{prop}

\begin{proof}
Let us write $\mathcal H=\HH\oplus \HH c$, where $c \in Z(\mathcal H),c^2=0$. 
Then we have $Im\,\rho(c)\subset Ker\,\rho(c)$. Both subspaces $Im\,\rho(c), Ker\,\rho(c)\subset V_\RR$
are $\HH$-invariant, so, choosing an $\HH$-invariant complement $U\subset V_\RR$ to $Ker\,\rho(c)$
we get that $\rho(c)$ induces an isomorphism $U \cong Im\,\rho(c)$. Set $4l=\dim\, U$.
Similarly, there is an
$\HH$-invariant complement to $Im\,\rho(c)$ in $Ker\,\rho(c)$, of dimension $4k$ for an appropriate $k$.
 Hence  
we can write $\rho=k\rho_4\oplus l\rho_8$. 
\end{proof}


\subsection{The proof of Theorem \ref{Main-theorem}}
\label{Main-theorem-subsection}
Let us deal first with the uniqueness statements.
The part of the statement of Theorem \ref{Main-theorem} for hyperbolic $\mathcal H\cong \mathcal H(-1)$
follows from the statement of Proposition \ref{prop-regular}, that there is exactly one irreducible representation $\rho_0$ of $\mathcal H$,
which is a cyclic $\HH$-module.  
 Hence an arbitrary $4n$-representation $\rho$ of $\mathcal H$ is isomorphic to $n\rho_0=\bigoplus^n \rho_0$.

The part of the statement for spherical $\mathcal H=\mathcal H(1)$ follows from the statement
of Proposition \ref{prop-regular}, that there are exactly two irreducible 4-representations $\rho_1,\rho_2$
of $\mathcal H$, which are non-faithful and correspond to 
factoring $\mathcal H \rightarrow \HH$ with respect to each of the two ideals found in this case. Hence, an arbitrary representation can be written as $\rho=
k\rho_1\oplus (n-k)\rho_2$ 
and the fact that $\rho_1(c)=Id_{\RR^4}$ and $\rho_2(c)=-Id_{\RR^4}$ explained in 
\ref{Spherical-irreps} tells us that 
$Tr(\rho_1(c))=4$ 
and $Tr\,\rho_2(c)=-4$, so that $Tr\,\rho(c)=4(2k-n)$, which uniquely identifies $k$ and hence multiplicities of both $\rho_1$
and $\rho_2$  in the decomposition of $\rho$. 

Note that $k=0,\dots,n$, where the extremal cases $k=0,n$
correspond to  non-faithful representations, so that there are total of $n+1$ non-equivalent 
representations of the spherical algebra $\mathcal H$.

In the case $\mathcal H\cong \mathcal H(0)$, by Proposition \ref{prop-regular}, the irreducible representations of the cylindrical algebra $\mathcal H$
are the 8-representation $\rho_8$ and the (non-faithful, factoring through $\HH$) 4-representation $\rho_4$. 
 Proposition \ref{Prop-special-representation} tells us 
that an arbitrary representation $\rho$ of $\mathcal H(0)$ is isomorphic to a sum of 
these representations, $\rho=k\rho_8 \oplus l\rho_4$, where $8k+4l=4n$. The number of all such possible non-equivalent representations, including the
trivial one, $\mathcal H \rightarrow \HH$, is $\lfloor \frac{n}{2}\rfloor+1$. It is also clear that
$rk\, \rho(c)=4k$, which uniquely identifies the multiplicities $k$ and $l$.

The existence part now follows for the case $a)$ $\mathcal H_{\alpha,\beta,\gamma}\cong\mathcal H(-1)$ 
from existence and faithfulness of the 4-dimensional representation $\rho_0$, 
and in the case $b)$ 
it follows from  
the obvious faithfulness of $\rho_{reg}\colon \mathcal H \rightarrow End\,\RR^8$.

\section{Proof of Theorem \ref{Theorem-relation}}
\label{Proof of Theorem-relation}

Let us sketch the plan of the proof.  Let $S\subset Compl \subset End\,V_\RR$
be a compact twistor line and $I_1,I_2,I_3 \in S$ be linearly independent 
complex structure operators, $G_\HH=G_{I_1}\cap G_{I_2} \cap G_{I_3}=G_{I_1}\cap G_{I_2}$. 
Set $(\alpha,\beta,\gamma){=}T(\triangle I_1 I_2 I_3)$. 
 As it was explained in the introduction,
given $(g_1,g_2,g_3)\in m^{-1}(G_\HH)$ we can construct a (possibly degenerate) twistor triangle 
$\triangle I_1 I_2J_3,J_3= g_2(I_3)=g_1^{-1}(I_3),$ formed by the lines $S=S(I_1,I_2),g_2(S),g_2g_3(S)$, $g_2(I_3)\in g_2(S) \cap g_2g_3(S)$. As $Tr(I_2I_3)=Tr(g_2(I_2I_3))=Tr(g_2(I_2)g_2(I_3))=Tr(I_2J_3)$
and, similarly, $Tr(I_1I_3)=Tr(g_1^{-1}(I_1I_3))=Tr(I_1J_3)$, we have $T(\triangle I_1 I_2 J_3)=(\alpha,\beta,\gamma)$,
so that there is a mapping $$\tau\colon m^{-1}(G_\HH) \rightarrow \mathcal T_{\alpha,\beta,\gamma},$$
$$(g_1,g_2,g_3)\mapsto \triangle I_1 I_2 g_2(I_3),$$
where $\mathcal T_{\alpha,\beta,\gamma}=\{\mbox{triangles }\triangle I_1I_2J_3\subset Compl \mbox{ with }T(\triangle I_1I_2J_3)=(\alpha,\beta,\gamma)\}$. Here $I_1,I_2$ are fixed. 

First, we can easily see that $\mathcal T_{\alpha,\beta,\gamma}$ is naturally a union of $G_\HH$-orbits,  
which are in one-to-one correspondence with the $n+1$
equivalence classes of representations 
$\rho\colon \mathcal H_{\alpha,\beta,\gamma}\rightarrow End\,V_\RR$ of the spherical algebra
$\mathcal H_{\alpha,\beta,\gamma}$. There are exactly two orbits in $\mathcal T_{\alpha,\beta,\gamma}$,
that are one-element sets, each of which consists of a degenerate triangle, supported on $S$, these orbits correspond to the two classes of nonfaithful representations. 

Second, we explicitly determine the fibers of $\tau$ over the two $G_\HH$-inequivalent degenerate triangles in the image of $\tau$,
supported on $S$:  these fibers are the trivial and the $SO(3)$-type
components of $m^{-1}(G_\HH)$ listed in the statement of Theorem \ref{Theorem-relation}. 

Third, we show that $\tau$ is onto, in particular, the connected components of 
$m^{-1}(G_\HH)$ reduce to those of the fibers $\tau^{-1}(G_\HH \cdot \triangle I_1 I_2 J_3)$ 
of $\tau$ 
over the $G_\HH$-orbits in $\mathcal T_{\alpha,\beta,\gamma}$. 

Fourth, we calculate the fibers of $\tau$ over individual nondegenerate triangles in $\mathcal T_{\alpha,\beta,\gamma}$.
Each such fiber  is shown to be diffeomorphic to $G_\HH\times G_\HH\times G_\HH$, being thus connected. 

Fifth, 
 we calculate the fibers of $\tau$ over the $G_\HH$-orbits $G_\HH \cdot \triangle I_1 I_2 J_3$ in $\mathcal T_{\alpha,\beta,\gamma}$, 
these fibers, being connected, are the connected components of $m^{-1}(G_\HH)$ described in Theorem \ref{Theorem-relation}. 

Finally, we do a calculation showing that for any triangle in the orbit, corresponding to the 
equivalence class of the balanced representation of $\mathcal H_{\alpha,\beta,\gamma}$,   
its angles, up to taking complements to $\pi$, are equal to the respective angles of 
the spherical triangle $\triangle I_1 I_2 I_3$.

\subsection{ The $G_\HH$-orbit structure of  $\mathcal T_{\alpha,\beta,\gamma}$}
%
%
By definition, for every triangle $\triangle I_1I_2J_3 \in \mathcal T_{\alpha,\beta,\gamma}$
we have $T(\triangle I_1I_2J_3)=T(\triangle I_1I_2I_3)=(\alpha,\beta,\gamma)$. 
Writing $T(\triangle I_1 I_2 I_3)=\frac{1}{4n}(Tr\, I_1I_2,Tr\, I_2I_3,Tr\, I_3I_1)=
(-\cos \sphericalangle I_1I_2,-\cos \sphericalangle I_2I_3,-\cos \sphericalangle I_3I_1)$, by the spherical
cosine law we have that $\cos\,\angle I_2I_1I_3=\frac{\cos \sphericalangle I_2I_3-
\cos \sphericalangle I_1I_2 \cos \sphericalangle I_1I_3}
{\sin \sphericalangle I_1I_2 \sin \sphericalangle I_1I_3}=-\frac{\beta+\alpha\gamma}
{\sqrt{1-\alpha^2}\sqrt{1-\gamma^2}}$, where by the $\angle I_2I_1I_3$ we mean the angle formed by the geodesic segments $I_1I_2$ and $I_1I_3$. 
Now the fact that the ratio on the right side is a cosine of a certain angle, 
the angle is not equal to $0$ or $\pi$,
means that $\left|\frac{\beta+\alpha\gamma}
{\sqrt{1-\alpha^2}\sqrt{1-\gamma^2}}\right|<1$ which is precisely the requirement that 
$det\, Q_{\alpha,\beta,\gamma}<0$. As, for example, 
$|\alpha|<1$, 
we see that $\mathcal H_{\alpha,\beta,\gamma}$ contains $\HH$,  
thus, the signature of $Q_{\alpha,\beta,\gamma}$ can only be $(0,3)$. Thus, 
 the algebra $\mathcal H_{\alpha,\beta,\gamma}$  is spherical. 

Next, Theorem \ref{Main-theorem} tells us that there are $n+1$ classes of $G$-equivalent representations
of $\rho\colon \mathcal H_{\alpha,\beta,\gamma} \rightarrow End\,V_\RR$, among which there are
 2 non-faithful representations, factoring through $\mathcal H_{\alpha,\beta,\gamma}\rightarrow \HH$ and $n-1$ faithful representations. Fix an embedding of $\HH\hookrightarrow End\,V_\RR$
which corresponds to the subalgebra in $End\,V_\RR$ generated by $I_1,I_2$,
$\HH\cong \langle Id,I_1,I_2,I_1I_2\rangle \subset End\,V_\RR$, and define the set
$$Rep_{I_1,I_2}(\mathcal H_{\alpha,\beta,\gamma})=
\{\mbox{representations }
\rho\colon \mathcal H_{\alpha,\beta,\gamma}\rightarrow End\,V_\RR,\\ \rho(e_1)=I_1,\rho(e_2)=I_2\}.$$
The group $G_\HH$ acts on $Rep_{I_1,I_2}(\mathcal H_{\alpha,\beta,\gamma})$
via the adjoint action 
and, as there are total
of $n+1$ $G$-equivalence classes of general representations $\rho \colon \mathcal H_{\alpha,\beta,\gamma} \rightarrow End\,V_\RR$, we see that $Rep_{I_1,I_2}(\mathcal H_{\alpha,\beta,\gamma})$
is a union of $n+1$ $G_\HH$-orbits, that are $G_\HH$-equivalence classes. 
%

Once we make a choice of a central element $c,c^2=1$, each $\rho \in Rep_{I_1,I_2}(\mathcal H_{\alpha,\beta,\gamma})$ is uniquely determined by 
the image $\rho(c)$ 
or, equivalently,  by  the  image $\rho(e_3)=J_3\in Compl$. For example,
if we set $c=\frac{1}{r}(\beta e_1 -\gamma e_2 +(\alpha-e_1e_2)e_3),r=\sqrt{-\det\,Q_{\alpha,\beta,\gamma}}$, then $c$ is such a normalized central element and 
$$\rho(e_3)=\rho\left( (\alpha-e_1 e_2)^{-1}(rc - \beta e_1 + \gamma e_2)\right)
=(\alpha Id-I_1 I_2)^{-1}(r\rho(c) - \beta I_1 + \gamma I_2) \in S.$$ 

Thus, there is a $G_\HH$-equivariant bijection between the sets $Rep_{I_1,I_2}(\mathcal H_{\alpha,\beta,\gamma})$ and $ \mathcal T_{\alpha,\beta,\gamma}$ 
given by $\rho\mapsto \triangle I_1 I_2\rho(e_3)$.

Two $G_\HH$-equivalence classes of nonfaithful representations lead to two
 degenerate triangles supported on the same sphere $S$, and are determined by the choice 
$\rho(c)=\pm Id$ 
 or, equivalently, by
the choice of $\rho(e_3)=(\alpha Id-I_1 I_2)^{-1}(\pm r Id - \beta I_1 + \gamma I_2) \in S,$
one of which is our fixed $I_3$, we denote the other one by $I_3^\prime$. The relation of $I_3^\prime$ to $I_3$ will be explained below. 
The $G_\HH$-equivalence classes of $n-1$ faithful representations lead to $n-1$ 
$G_\HH$-equivalence classes of non-degenerate triangles $\triangle I_1 I_2 J_3 \in \mathcal T_{\alpha,\beta,\gamma}$. 

\subsection{The trivial component $\tau^{-1}(\triangle I_1 I_2 I_3)$}
For any $(g_1,g_2,g_3) \in \tau^{-1}(\triangle I_1 I_2 I_3)$ we must have $I_3=g_2(I_3)$,
that is, $g_2\in G_{I_2} \cap G_{I_3}=G_\HH$. Next, the requirement $g_1g_2g_3 \in G_\HH$
tells us that $I_3=g_1g_2g_3(I_3)=g_1g_2(I_3)=g_1(I_3)$, that is $g_1\in G_{I_1}\cap G_{I_3}=G_\HH$
and then, clearly, $g_3 \in G_\HH$. That is, we have the inclusion $\tau^{-1}(\triangle I_1 I_2 I_3)
\subset G_\HH\times G_\HH \times G_\HH$. The inclusion $G_\HH\times G_\HH \times G_\HH
\subset \tau^{-1}(\triangle I_1 I_2 I_3)$ is obvious, so that we have 
$\tau^{-1}(\triangle I_1 I_2 I_3)=G_\HH\times G_\HH \times G_\HH$, the most trivial component of 
$m^{-1}(G_\HH)$.

\subsection{The $SO(3)$-type component $\tau^{-1}(\triangle I_1 I_2 I_3^\prime)$.}
For a triple $(g_1,g_2,g_3) \in \tau^{-1}(\triangle I_1 I_2 I_3^\prime)$
we must have $g_2(I_3)=I_3^\prime \in S$, that is, $g_2$ takes 
$S=S(I_2,I_3)$ to $S=S(I_2,I_3^\prime)$, so that  $g_2\in G_{I_2,S}$, moreover,  
$g_2$ is determined uniquely, up to an element in $G_\HH$.
Then writing $g_1g_2g_3=h \in G_\HH$ we see that $I_3=g_1g_2g_3(I_3)=g_1(I_3^\prime)$,
that is, $g_1$ takes $I_3^\prime \in S$ to $I_3\in S\setminus \{\pm I_1\}$. This, together with
$g_1(I_1)=I_1$  implies that $g_1$ takes $S=S(I_1,I_3)$ to $S=S(I_1,I_3^\prime)$,
that is, $g_1 \in G_{I_1,S}$. Again, such $g_1$ is determined uniquely, up to an element in $G_\HH$.
Finally, $g_3$ must also take $S$ to $S$, so that $g_3\in G_{I_3,S}$ and again, due to the relation 
$g_1g_2g_3 \in G_\HH$, it is determined uniquely, 
up to an element in $G_\HH$.

%
%
The element $g_j$ as above acts as a rotation of $S$ about the corresponding axis $\{\pm I_j\},j=1,2,3$. 
The subgroups of (isometric) rotations $\langle e^{tI_1}|\,t\in\RR\rangle \subset G_{I_1}, 
\langle e^{tI_2}|\, t\in \RR \rangle \subset G_{I_2},\langle e^{tI_3}|\,t\in\RR \rangle  \subset G_{I_3}$
generate a subgroup in $G_S$ isomorphic to $SO(3)$. 
It is clear that fixing the (unique) rotations $e^{t_1I_1}, e^{t_2I_2}$, such that
$I_3^{\prime}=e^{t_1I_1}(I_3)=e^{t_2I_2}(I_3)$, we get that $e^{-t_1I_1}e^{t_2I_2}$ is a rotation in $SO(3)\subset G_S$ 
about the axis $\{\pm I_3\}$. Then we can find
a unique $e^{t_3I_3}$, such that $e^{-t_1I_1}e^{t_2I_2}e^{t_3I_3}=1 \in G_\HH$.

\vspace*{1cm}
\hspace*{4cm}
\input{SO3-corr.pic}
\vspace*{-12.5cm}
\begin{center}
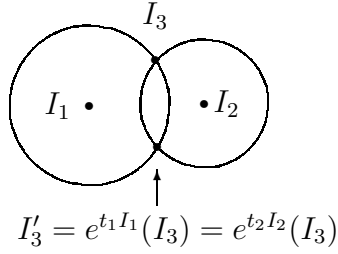
Picture 2: Obtaining a relation among rotations around $I_1,I_2,I_3$.
\end{center}

Thus, allowing $g_1,g_2,g_3$ to be defined up to elements in $G_\HH$, we get
that $\tau^{-1}(\triangle I_1 I_2 I_3^\prime)=e^{-t_1I_1}G_\HH \times e^{t_2I_2}G_\HH
\times e^{t_3I_3}G_\HH$. This is the $SO(3)$-type component from the statement
 of Theorem \ref{Theorem-relation}. 

\subsection{The surjectivity of $\tau \colon m^{-1}(G_\HH) \to \mathcal T_{\alpha,\beta,\gamma}$}

Given a triangle $\triangle I_1 I_2 J_3 \in \mathcal T_{\alpha,\beta,\gamma}$ with
$T( \triangle I_1 I_2 J_3)=T(\triangle  I_1 I_2 I_3)=(\alpha,\beta,\gamma)$, 
we need to find a triple  $(g_1,g_2,g_3) \in m^{-1}(G_\HH)$, such that 
$\tau(g_1,g_2,g_3)= \triangle I_1 I_2 J_3$. 

As we discussed in the introduction, we can find $g_1\in G_{I_1},g_2\in G_{I_2}$ 
such that 
$g_2(S)=S(I_2,J_3)$ and $g_1^{-1}(S)=S(I_1,J_3)$. 
As $G_{J_3}$ acts transitively on the set of twistor lines containing 
$\pm J_3$ (see \cite{Twistor-lines}), we can take $f_3\in G_{J_3}$, such that $f_3(S(I_2,J_3))=S(I_1,J_3)$
(here we do not assume that $f_3$ takes $I_2$ to $I_1$, it is merely an equality of sets), and set $g_3=g_2^{-1}f_3g_2 \in 
G_{g_2^{-1}(J_3)}$. Then, obviously $g_1g_2g_3(S)=S$. 

Here, certainly, $g_2^{-1}(J_3)$ need not be equal the initially fixed $I_3$ and while the product
$g_1g_2g_3$ belongs to $G_S$, it need not be in $G_\HH$. We want to modify the triple $(g_1,g_2,g_3)$
so as to satisfy conditions $g_1g_2g_3 \in G_\HH$, $g_j\in G_{I_j},j=1,2,3$, and 
$\tau(g_1,g_2,g_3)=\triangle I_1 I_2 J_3$. 

Set $\widetilde{I_3}=g_2^{-1}(J_3)$ and set $\widetilde{\widetilde{I_3}}=g_1(J_3)$.
We clearly have that $\arccos(-\beta)=d(I_2,I_3)=d(I_2,\widetilde{I_3})\mbox{ and }
\arccos(-\gamma)=d(I_1,I_3)=d(I_1,\widetilde{\widetilde{I_3}}),$
where $d(\cdot,\cdot)$ denotes the (spherical) distance between the points of $S$, see the picture below.

Now let us choose $e^{t_2I_2}\in G_{I_2,S}$ and $e^{t_1I_1} \in G_{I_1,S}$ such that $e^{t_2I_2}(\widetilde{I_3})=I_3 \mbox{ and }e^{t_1I_1}(\widetilde{\widetilde{I_3}})=I_3.$ Then the modified elements $e^{t_1I_1}g_1\in G_{I_1},g_2e^{-t_2I_2}\in G_{I_2}$
still  lead to the same triple of consecutive twistor lines as $g_1,g_2$ did: $(e^{t_1I_1}g_1)^{-1}(S)=g_1^{-1}(e^{-t_1I_1}(S))=g_1^{-1}(S)=S(I_1,J_3)$
and $g_2e^{-t_2I_2}(S)=g_2(S)=S(I_2,J_3)$. Now we have that $e^{t_2I_2}g_3e^{-t_2I_2}\in G_{I_3}$,
indeed: $e^{t_2I_2}g_3e^{-t_2I_2}(I_3)=e^{t_2I_2}g_3(\widetilde{I_3})=e^{t_2I_2}(\widetilde{I_3})=I_3$.
 Obviously the element $e^{t_1I_1}g_1 \cdot g_2e^{-t_2I_2}\cdot e^{t_2I_2}g_3e^{-t_2I_2}$ 
takes $S$ to itself,
moreover,  
$e^{t_1I_1}g_1 \cdot g_2e^{-t_2I_2}\cdot e^{t_2I_2}g_3e^{-t_2I_2}(I_3)=
e^{t_1I_1}g_1 \cdot g_2e^{-t_2I_2}(I_3)=
e^{t_1I_1}g_1 \cdot g_2(\widetilde{I_3})=e^{t_1I_1}(\widetilde{\widetilde{I_3}})=I_3$. This means
that our product $e^{t_1I_1}g_1 \cdot g_2e^{-t_2I_2}\cdot e^{t_2I_2}g_3e^{-t_2I_2}=e^{t_3I_3}h
= he^{t_3 I_3} \in G_{I_3,S}$ for an appropriate $h \in G_\HH$. 
 So for the ``corrected'' elements $e^{t_1I_1}g_1 \in G_{I_1},g_2e^{-t_2I_2} \in G_{I_2},
e^{t_2I_2}g_3e^{-t_2I_2}e^{-t_3I_3}\in G_{I_3}$ their product   $(e^{t_1I_1}g_1) \cdot (g_2e^{-t_2I_2})\cdot (e^{t_2I_2}g_3e^{-t_2I_2}e^{-t_3I_3})=h$ belongs to $G_\HH$
and they map under $\tau$ to the twistor triangle $\triangle I_1 I_2 J_3$. 

That is, for an arbitrary twistor triangle $\triangle I_1I_2J_3 \in \mathcal T_{\alpha,\beta,\gamma}$ 
 we can find $(g_1,g_2,g_3)\in m^{-1}(G_\HH)$ such that $\tau(g_1,g_2,g_3)=\triangle I_1I_2J_3$. 

\vspace*{0.2cm}
\hspace*{2cm}
\input{triangle-relation-2.pic}
\vspace*{-10.50cm}
\begin{center}
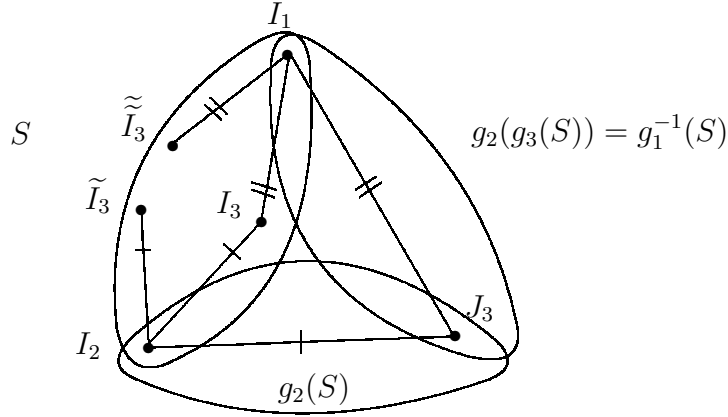
Picture 3: modifying the triple $(g_1,g_2,g_3)$ in order to get $g_1g_2g_3\in G_\HH$.
\end{center}
%

\subsection{The $G_\HH^3$-structure of  the fiber $\tau^{-1}(\triangle I_1 I_2 J_3)$.} 
\label{Subsec-individual-fiber}
Given an arbitrary 
nondegenerate triangle $\triangle I_1 I_2 J_3 \in \mathcal T_{\alpha,\beta,\gamma}$ we want to describe $\tau^{-1}(\triangle I_1 I_2 J_3)$. 

As we already know, $\tau^{-1}(\triangle I_1 I_2 J_3)$ is non-empty, so let us choose
$(f_1,f_2,f_3)\in \tau^{-1}(\triangle I_1 I_2 J_3)$. 
Let us now show that $$\tau^{-1}(\triangle I_1 I_2 J_3)=\{(h_1f_1, f_2h_2^{-1}, h_2f_3h_3^{-1})|h_1,h_2,h_3\in G_\HH\}\cong G_\HH^3,$$
the orbit under the action of $G_\HH^3$ defined by 
$(h_1,h_2,h_3)\cdot (f_1,f_2,f_3)=(h_1f_1, f_2h_2^{-1}, h_2f_3h_3^{-1})$. 
The inclusion `$\supset$' is obvious, so we only need to show the inclusion `$\subset$',
that is, given $(g_1,g_2,g_3)\in \tau^{-1}(\triangle I_1 I_2 J_3)$ we need to find $h_1,h_2,h_3 \in G_\HH$
such that $(g_1,g_2,g_3)=(h_1f_1, f_2h_2^{-1}, h_2f_3h_3^{-1})$.


We recall that, as every twistor line is uniquely identified by its any two non-proportional points,
the condition that $(f_1,f_2,f_3),(g_1,g_2,g_3)\in \tau^{-1}(\triangle I_1 I_2 J_3)$, implying that
$f_2(I_3)=g_2(I_3)$ tells us that we have set-theoretic equalities $f_2(S)=g_2(S), f_2(f_3(S))=g_2(g_3(S))$. 
We note that as $f_1^{-1}(S)=f_2(f_3(S))=g_2(g_3(S))=g_1^{-1}(S)$, 
we have that $g_1f_1^{-1}(S)=S$ and as $g_1f_1^{-1} \in G_{I_1}$ we have that
$g_1=h_1e^{t_1I_1}f_1=h_1f_1e^{t_1I_1}$ for $h_1\in G_\HH$ and an appropriate $t_1 \in \RR$. 
Next, we certainly have that $f_2(S)=g_2(S)$, 
which analogously gives that  $g_2^{-1}f_2(S)=S$, hence $g_2^{-1}f_2=e^{t_2I_2}h_2$
for an appropriate $h_2\in G_\HH$ and $t_2\in \RR$,                
so that $g_2=f_2h_2^{-1}e^{-t_2I_2}$.
At the same time, as $\tau(f_1,f_2,f_3)=\tau(g_1,g_2,g_3)$ and so $f_2(I_3)=g_2(I_3)$, 
we see that $g_2^{-1}f_2(I_3)= I_3$,
and then $g_2^{-1}f_2 \in G_{I_2}\cap G_{I_3}=G_\HH$, that is, $t_2=0$ and $g_2=f_2h_2^{-1}$,

Now let us figure how much freedom the choice of $g_3$ has with respect to $f_3$.
For that we look at the set-theoretic equality $f_2(f_3(S))=g_2(g_3(S))$
and use that $g_2=f_2h_2^{-1}$.  

The equality  $g_2(g_3(S))=f_2h_2^{-1}(g_3(S))=f_2(f_3(S))$
implies that $g_3^{-1}h_2f_3= e^{t_3I_3}h_3\in G_{I_3,S}$ for appropriate $t_3 \in \RR$ and $h_3\in G_\HH$, so that $g_3=h_2f_3h_3^{-1}e^{-t_3I_3}$. 
The condition $(f_1,f_2,f_3),(g_1,g_2,g_3)\in m^{-1}(G_\HH)$ 
translates into the equality of mappings $Id_S=g_1g_2g_3=h_1f_1e^{t_1I_1}\cdot f_2h_2^{-1} \cdot h_2f_3h_3^{-1}e^{-t_3I_3}=h_1f_1e^{t_1I_1}\cdot f_2f_3h_3^{-1}e^{-t_3I_3}$. 
Now, as $f_1f_2f_3=h\in G_\HH$ we have that $f_2f_3=f_1^{-1}h$
so that replacing $f_2f_3$ with this expression in the formula for $g_1g_2g_3$
we get that $g_1g_2g_3=h_1f_1e^{t_1I_1}\cdot f_1^{-1}hh_3^{-1}\cdot e^{-t_3I_3}
=h_1hh_3^{-1}e^{t_1I_1}e^{-t_3I_3}=Id_S$. From this we see that
we must already have $e^{t_1I_1}e^{-t_3I_3}=Id_S$, so that 
$e^{t_1I_1}=Id_S,e^{t_3I_3}=Id_S$, which immediately specifies $t_1=t_3=0$. 
So finally $(g_1,g_2,g_3)=(h_1f_1, f_2h_2^{-1}, h_2f_3h_3^{-1})$, 
concluding that the set $\tau^{-1}(\triangle I_1 I_2 J_3)$ is diffeomorphic to 
$G_\HH \times G_\HH \times G_\HH$. 

\medskip

\subsection {The fiber $\tau^{-1}(G_\HH\cdot \triangle I_1 I_2 J_3)$}
For any $g \in G_\HH$ and any $\triangle I_1 I_2 J_3 \in \mathcal T_{\alpha,\beta,\gamma}$ 
we have $g\cdot \triangle I_1 I_2 J_3=\triangle g(I_1)g(I_2)g(J_3)
=\triangle I_1 I_2 g(J_3)$, 
so that the orbit $G_\HH\cdot \triangle I_1 I_2 J_3$ is the subset $\{\triangle I_1 I_2 g(J_3)\,|\,g \in G_\HH\}
\subset \mathcal T_{\alpha,\beta,\gamma}$. 
As we have seen in \ref{Subsec-individual-fiber}, the preimage of an individual triangle 
$\tau^{-1}(\triangle I_1 I_2 J_3)=
\{(h_1f_1, f_2h_2^{-1}, h_2f_3h_3^{-1})|h_1,h_2,h_3\in G_\HH\}$, where $(f_1,f_2,f_3)$ 
is an arbitrary point in $\tau^{-1}(\triangle I_1 I_2 J_3)$, has an obvious structure
of the homogeneous space diffeomorphic to $G_\HH^3$. 
The preimage  $\tau^{-1}(g\cdot \triangle I_1 I_2 J_3), g \in G_\HH,$ contains the point 
$g\cdot (f_1,f_2,f_3)=(gf_1g^{-1},gf_2g^{-1},gf_3g^{-1})$
and so, again by \ref{Subsec-individual-fiber}, we can write $$\tau^{-1}(g\cdot \triangle I_1 I_2 J_3)=\{(h_1gf_1g^{-1}, gf_2g^{-1}h_2^{-1}, h_2gf_3g^{-1}h_3^{-1})|h_1,h_2,h_3\in G_\HH\}=$$ $$=
\{(h_1f_1g^{-1}, gf_2h_2^{-1}, h_2f_3h_3^{-1})|h_1,h_2,h_3\in G_\HH\}.$$
Then we have $$\tau^{-1}(G_\HH\cdot \triangle I_1 I_2 J_3)=\underset{g\in G_\HH}{\bigcup}
\tau^{-1}(g\cdot \triangle I_1 I_2 J_3)=$$ $$
=\{(h_1f_1g^{-1}, gf_2h_2^{-1}, h_2f_3h_3^{-1})|g,h_1,h_2,h_3\in G_\HH\},$$
so that the preimage $\tau^{-1}(G_\HH\cdot \triangle I_1 I_2 J_3)$ has an obvious structure of
a homogeneous $G^4_\HH$-manifold,
namely the action is defined by $$(h_1,g,h_2,h_3)\cdot (f_1,f_2,f_3)=(h_1f_1g^{-1}, gf_2h_2^{-1}, h_2f_3h_3^{-1}).$$ Thus, up to renaming the entries of tuples in $G_\HH^4$, it is a connected subset in $m^{-1}(G_\HH)$ of the form specified in
Theorem \ref{Theorem-relation}. In order to
determine the diffeomorphism type of the orbit $G_\HH^4\cdot (f_1,f_2,f_3)$ for an arbitrary $(f_1,f_2,f_3)\in \tau^{-1}(G_\HH\cdot \triangle I_1 I_2 J_3)$ we need to find the stabilizer of that point.  
The stabilizer, by definition of the action,
consists of those 4-tuples $(h_1,g,h_2,h_3)\in G_\HH^4$ which satisfy
$(h_1,g,h_2,h_3)\cdot (f_1,f_2,f_3)=(h_1f_1g^{-1}, gf_2h_2^{-1}, h_2f_3h_3^{-1})=(f_1,f_2,f_3)$.

Equating the first entries, we get  $h_1f_1g^{-1}=f_1$ or $g=f_1^{-1}h_1f_1$, thus
$g \in G_\HH \cap f_1^{-1}G_\HH f_1$. As $f_1^{-1}(I_1)=I_1$, $f_1^{-1}(I_3)=f_2f_3(I_3)=f_2(I_3)=J_3$, we see that $g$ must pointwise stabilize $I_1,I_2$ and $J_3=f_1^{-1}(I_3)$,
so that in the end it stabilizes pointwise the subalgebra $\mathcal H(I_1,I_2,J_3)\subset End\,V_\RR$
under the $G$-action. 
Further we denote the pointwise $G$-action stabilizer of  $\mathcal H(I_1,I_2,J_3)$ by $G_{\mathcal H(I_1,I_2,J_3)}\subset G_\HH$. 
For any $g \in G_{\mathcal H(I_1,I_2,J_3)}$ we can determine, in a unique way, 
the respective $h_1=f_1gf_1^{-1}\in G_\HH$.

Equating the second entries we get $gf_2h_2^{-1}=f_2$, that is, $g\in G_\HH \cap f_2 G_\HH f_2^{-1}$. This condition is equivalent to $g$ stabilizing pointwise $I_1,I_2$ and $J_3=f_2(I_3)$,
that is, $g\in  G_{\mathcal H(I_1,I_2,J_3)}$, so that no further restriction is added, and, again, for any such $g$ we uniquely determine 
$h_2=f_2^{-1}g f_2 \in G_\HH$. 

Equating the third entries and using the previously determined $h_2$ we get that $h_3=f_3^{-1}h_2f_3=f_3^{-1}\cdot f_2^{-1}g f_2 \cdot f_3$.
As we know that $(f_1,f_2,f_3)\in m^{-1}(G_\HH)$, so that $f_1f_2f_3=h\in G_\HH$, we
get $f_2f_3=f_1^{-1}h$, and then, using the previously determined $h_1$,  
we get $h_3=h^{-1} \cdot f_1 g f_1^{-1}\cdot  h=h^{-1}h_1h\in G_\HH$. 

Thus, $Stab_{G_\HH^4}(f_1,f_2,f_3)=\{(g,f_1gf_1^{-1}, f_2^{-1}g f_2, f_3^{-1}f_2^{-1}g f_2f_3)\,|\, g \in G_{\mathcal H(I_1,I_2,J_3)}\} \cong G_{\mathcal H(I_1,I_2,J_3)}$  
and $\tau^{-1}(G_\HH\cdot \triangle I_1 I_2 J_3)\cong G_\HH^4/G_{\mathcal H(I_1,I_2,J_3)}$,
here we identify $G_{\mathcal H(I_1,I_2,J_3)}$ with its image in $G_\HH^4$
under the above specified isomorphism $G_{\mathcal H(I_1,I_2,J_3)}\cong Stab_{G_\HH^4}(f_1,f_2,f_3)$. 

Now $\mathcal T_{\alpha,\beta,\gamma}$ is a union of $n+1$ distinct orbits 
of the form $G_\HH\cdot \triangle I_1 I_2 J_3$. As $G_\HH$ is connected, 
and each of the fibers of $\tau$ over such an orbit is diffeomorphic to $ G_\HH^4/G_{\mathcal H(I_1,I_2,J_3)}$, we see that the fibers of $\tau$ are connected subsets of $m^{-1}(G_\HH)$. 

\subsection{The connected components of $m^{-1}(G_\HH)$}

First, we want to show  that the 
 (proved to be connected) fibers $\tau^{-1}(G_\HH \cdot \triangle I_1 I_2 J_3)$
are the connected components of $m^{-1}(G_\HH)$.  Second,
we will calculate the dimension of the connected components. 

For that it is sufficient to show that the $G_\HH$-orbits form the set of connected 
components of $\mathcal T_{\alpha,\beta,\gamma}$. This  
follows from the fact that $G_\HH$-orbits form the set of connected components of 
the topological space $Rep_{I_1,I_2}(\mathcal H_{\alpha,\beta,\gamma})$, which is 
$G_\HH$-equivariantly isomorphic to  $\mathcal T_{\alpha,\beta,\gamma}$, the isomorphism
$Rep_{I_1,I_2}(\mathcal H_{\alpha,\beta,\gamma}) \cong \mathcal T_{\alpha,\beta,\gamma}$ is
given by $\rho \mapsto \triangle I_1 I_2 \rho(I_3)$. 
Indeed,  each representation $$\rho \in Rep_{I_1,I_2}(\mathcal H_{\alpha,\beta,\gamma}), \rho\colon \mathcal H_{\alpha,\beta,\gamma} \rightarrow End\,V_\RR,   
\rho(e_1)=I_1, \rho(e_2)=I_2,$$ is uniquely determined by $\rho(e_3)$, or, which is equivalent, by
$\rho(c)=Id_{\RR^{4k}}\oplus -Id_{\RR^{4(n-k)}}$, see Theorem \ref{Main-theorem}, where $c$ is a the choice of a central element, satisfying $c^2=1$, 
$k=\frac{1}{8}(Tr(\rho(c))+4n)$ (such representation $\rho$ is $G_\HH$-isomorphic to $\rho_k$ defined in the introduction). 
Any two representations $\rho_1,\rho_2 \in Rep_{I_1,I_2}(\mathcal H_{\alpha,\beta,\gamma})$
are $G_\HH$-equivalent if and only if $Tr(\rho_1(c))=Tr(\rho_2(c))$. Thus 
the set $Rep_{I_1,I_2}(\mathcal H_{\alpha,\beta,\gamma})$ is a disjoint union of
its $n+1$ closed $G_\HH$-orbits $$\{\rho\in Rep_{I_1,I_2}(\mathcal H_{\alpha,\beta,\gamma})\,|\,Tr(\rho(c))=4(2k-n)\},k=0,\dots,n.$$ Hence, as $G_\HH$ is connected, the $G_\HH$-orbits form the set of connected components
of $\mathcal T_{\alpha,\beta,\gamma}\cong Rep_{I_1,I_2}(\mathcal H_{\alpha,\beta,\gamma})$,
therefore the fibers $\tau^{-1}(G_\HH \cdot \triangle I_1 I_2 J_3)$ are the connected components of $m^{-1}(G_\HH)$. 
Besides that, from the identification 
$Rep_{I_1,I_2}(\mathcal H_{\alpha,\beta,\gamma}) \cong \mathcal T_{\alpha,\beta,\gamma}$ we obtain that each connected component of 
$m^{-1}(G_\HH)=\tau^{-1}(\mathcal T_{\alpha,\beta,\gamma})$ has the stated form $$
\{(f_1,f_2,f_3)\in m^{-1}(G_\HH)\,|\,f_2(I_3) \in G_\HH \cdot \rho_k(e_3)\}$$
for an appropriate $k$, $0\leqslant k \leqslant n$.

The dimension of $ G_\HH^4/G_{\mathcal H(I_1,I_2,J_3)}$ is determined
from the fact that $g \in G$ stabilizes the subalgebra $\mathcal H(I_1,I_2,J_3) \subset End\, V_\RR$
if and only if it stabilizes 
the representation $\rho\colon \mathcal H_{\alpha,\beta,\gamma} \rightarrow End\,V_\RR$ given by  
$\rho(e_1)=I_1, \rho(e_2)=I_2,\rho(e_3) =  J_3$, or, equivalently, $g$ stabilizes $I_1,I_2$
together with  $\rho(c)=Id_{\RR^{4k}}\oplus -Id_{\RR^{4(n-k)}}$. 

The condition of centralizing $\rho(c)$ cuts out the subgroup $G_{\mathcal H(I_1,I_2,J_3)}= G_{\HH,\rho_k}$ of dimension $4k^2+4(n-k)^2$ in $G_\HH$. 
That is,  $\dim\,G_\HH^4/G_{\mathcal H(I_1,I_2,J_3)}=4 \cdot 4n^2-(4k^2+4(n-k)^2)=12n^2+8nk-8k^2$.
This dimension takes its smallest value $12n^2=\dim\, G_\HH^3$ exactly when $k=0$ or $k=n$,
that is, when $J_3=I_3$ or $J_3=I_3^\prime$. 

\bigskip
What is left now is the calculation comparing the angles of the triangle $\triangle I_1 I_2J_3$
to those of $\triangle I_1 I_2 I_3$. 
\subsection{Triangles: comparison of angles.}

Given a twistor triangle $\triangle I_1 I_2 J_3$
with $T(\triangle I_1 I_2 J_3)=(\alpha,\beta,\gamma)$, we have the relation between the 
corresponding complex structures: $I_1I_2+I_2I_1=2\alpha Id$, $I_1J_3+J_3I_1=2\gamma Id$, 
$I_2J_3+J_3 I_2=2\beta Id$. 
The complex structures $\frac{\alpha I_1+I_2}{\sqrt{1-\alpha^2}}\in S(I_1,I_2)$
and $\frac{\gamma I_1+J_3}{\sqrt{1-\gamma^2}}\in S(I_1,J_3)$ anticommute with
 $I_1$. Then we can write $$T_{I_1}S(I_1,I_2)=\left \langle \frac{\alpha I_1+I_2}{\sqrt{1-\alpha^2}}, 
I_1\frac{\alpha I_1+I_2}{\sqrt{1-\alpha^2}}\right \rangle, T_{I_1}S(I_1,J_3)=\left\langle \frac{\gamma I_1+J_3}{\sqrt{1-\gamma^2}}, 
I_1\frac{\gamma I_1+J_3}{\sqrt{1-\gamma^2}}\right\rangle.$$
This spaces obviously have trivial intersection. Set $W=T_{I_1}S(I_1,I_2) \oplus T_{I_1}S(I_1,J_3)$.
Consider the values of the form $q(x,y)=-\frac{1}{4n}Tr(x\cdot y) \colon W \times W \rightarrow \RR$
for the unit vectors in the tangent planes. Setting $c_t=\cos t \cdot \frac{\alpha I_1+I_2}{\sqrt{1-\alpha^2}}+ \sin t \cdot \frac{-\alpha Id+I_1I_2}{\sqrt{1-\alpha^2}}$
and $d_s=\cos s \cdot \frac{\gamma I_1+J_3}{\sqrt{1-\gamma^2}}+ \sin s \cdot 
\frac{-\gamma Id+I_1J_3}{\sqrt{1-\gamma^2}}$ we consider
$q(c_t, d_s),$
which is equal to the trace of the following
$$-\frac{1}{4n\sqrt{1-\alpha^2}\sqrt{1-\gamma^2}}\biggl(\cos t\cos s\cdot (-\alpha\gamma Id+\alpha I_1J_3$$
$$+\gamma I_2I_1 +I_2J_3)+$$
$$+\cos t\sin s (-\alpha\gamma I_1 -\alpha J_3-\gamma I_2+I_2I_1J_3)+$$
$$+\sin t\cos s(-\alpha\gamma I_1 -\alpha J_3+\gamma I_1 I_2 I_1 + I_1 I_2 J_3)+$$
$$+\sin t \sin s(\alpha\gamma Id-\alpha I_1 J_3 -\gamma I_1I_2+I_1 I_2 I_1 J_3)\biggr).$$

We know that traces of complex structures are zeroes and we know traces of all products of pairs
of distinct complex structures, like $I_1I_2, I_1J_3, I_2J_3$. The trace of $I_1I_2I_1=-I_1I_2I_1^{-1}$
is zero, and we have  $I_2I_1J_3=(-I_1I_2+2\alpha Id)J_3=-I_1I_2J_3+2\alpha J_3$, 
$I_1I_2I_1J_3=I_1(-I_1I_2+2\alpha Id)J_3=I_2J_3+2\alpha I_1 J_3$,
so what remains to calculate is the trace of $I_1 I_2 J_3$.

As in general the natural representation  $\rho\colon \mathcal H=\mathcal H(I_1,I_2,J_3) \rightarrow
End\,V_\RR$ decomposes, by Theorem \ref{Main-theorem}, as $\rho=k \rho_1\oplus l \rho_2$ for certain $k,l$, setting $m=4(l-k)$ and $c=\beta e_1-\gamma e_2 +\alpha e_3-e_1e_2e_3, r=\sqrt{-\det\,Q_{\alpha,\beta,\gamma}}$, so that $\rho_1(c)=rId_{\RR^4}, \rho_2(c)=-rId_{\RR^4}$, 
we have that
$Tr(I_1 I_2 J_3)=-Tr(\beta I_1-\gamma I_2 +\alpha J_3-I_1I_2J_3)=-Tr(k\rho_1(c)\oplus l\rho_2(c))=mr$.  

Now we calculate the above value of $q(c_t, d_s)$, $$q(c_t, d_s)=-\frac{1}{\sqrt{1-\alpha^2}\sqrt{1-\gamma^2}}
\left((\beta+\alpha\gamma)\cos(t-s)+ \frac{m}{4n}r\cdot\sin(t-s)\right).$$
Next,  if $m=0$, which is equivalent to $\rho$ being balanced, then $|q(c_t, d_s)|< 1$,
(so that the form $q$ is indeed positively definite on $W$) and the maximal value 
of $|q(c_t, d_s)|$ is attained for $t=s$,
it is equal to $\left|\frac{\beta+\alpha\gamma}{\sqrt{1-\alpha^2}\sqrt{1-\gamma^2}}\right|$, 
which, by the spherical cosine law, equals $\pm \cos \angle I_2I_1I_3$. The analogous computations can be done for the other two angles of $\triangle I_1I_2J_3$ showing that these angles are well defined and are equal to
the respective angles of $\triangle I_1 I_2 I_3$, up to taking complements to $\pi$.


\section{Appendix}
\label{Appendix}
In this section we prove the most nontrivial parts of the statement
of Proposition \ref{Prop-properties-of-H}.

\subsection{$S|_V=q|_V,  S|_{\widetilde{V}}=q|_{\widetilde{V}}$}



The part $S|_V=q|_V$ is really trivial and follows from the definition of $q$ and the relations
of the algebra $\mathcal H$.

Let us get to showing $S|_{\widetilde{V}}=q|_{\widetilde{V}}$. 
We easily see that 
$S(\alpha-e_1e_2)=\alpha^2-2\alpha e_1e_2+e_1 (e_2e_1)e_2=
\alpha^2-2\alpha e_1e_2+e_1 (-e_1e_2+2\alpha)e_2=\alpha^2-1=
q(\alpha-e_1e_2)$ and similarly for other basis elements of $\widetilde{V}$. 
We need to check the equality $S(v)=q(v)$ for the general elements $v\in \widetilde{V}$, for which now it  suffices
to check that mixed symmetric products of the kind $(\alpha-e_1e_2)(\beta-e_2e_3)+(\beta-e_2e_3)(\alpha-e_1e_2)$ land in $\RR\cdot Id$. Indeed,
$(\alpha-e_1e_2)(\beta-e_2e_3)+(\beta-e_2e_3)(\alpha-e_1e_2)=
2\alpha\beta-2\beta e_1e_2 -2 \alpha e_2 e_3 +e_1e_2\cdot e_2e_3+e_2e_3\cdot e_1e_2
=
2\alpha\beta-2\beta e_1e_2 -2 \alpha e_2 e_3 -e_1e_3+(-e_3e_2+2\beta) e_1e_2=
2\alpha\beta -2 \alpha e_2 e_3 -e_1e_3-e_3e_2e_1e_2=
2\alpha\beta -2 \alpha e_2 e_3 -e_1e_3-e_3(-e_1e_2+2\alpha)e_2=
2\alpha\beta -2 \alpha (e_2 e_3+e_3e_2) -e_1e_3-e_3e_1=2\alpha\beta-4\alpha\beta-2\gamma=
-2(\alpha\beta+\gamma) \in \RR$ and similarly for other pairs of basis elements.

\medskip

\subsection{ $Tr\, \rho_{reg}(e_1e_2e_3)=0$}
\label{subsec-triple-trace}
For the trace calculation we consider the basis $1,e_1,e_2,e_3,e_1e_2,e_2e_3,e_3e_1,e_1e_2e_3$ of $\mathcal H$.
Then 

$e_1e_2e_3\cdot 1=e_1e_2e_3$, 

$e_1e_2e_3\cdot e_1=e_1e_2(-e_1e_3+2\gamma)=-e_1(-e_1e_2+2\alpha)e_3+2\gamma e_1e_2
=-2\alpha e_1e_3 - e_2e_3+2\gamma e_1e_2$,

$e_1e_2e_3\cdot e_2=e_1e_2(-e_2e_3+2\beta)=e_1e_3+2\beta e_1e_2$,

$e_1e_2e_3\cdot e_3=-e_1e_2$,

$e_1e_2e_3\cdot e_1e_2=e_1e_2(-e_1e_3+2\gamma)e_2=-e_1(-e_1e_2+2\alpha)e_3e_2-2\gamma e_1
=-e_2e_3e_2-2\alpha e_1e_3e_2-2\gamma e_1=-(-e_3e_2+2\beta)e_2-
2\alpha e_1(-e_2e_3+2\beta)-2\gamma e_1=-(2\gamma+4\alpha\beta) e_1-2\beta e_2-e_3+2\alpha e_1e_2e_3$,

$e_1e_2e_3\cdot e_2e_3=e_1e_2(-e_2e_3+2\beta)e_3=-e_1+2\beta e_1e_2e_3$,

$e_1e_2e_3\cdot e_3e_1=-e_1(-e_1e_2+2\alpha)=-2\alpha e_1-e_2$,

$e_1e_2e_3\cdot e_1e_2e_3=e_1e_2(-e_1e_3+2\gamma)e_2e_3=-e_1(-e_1e_2+2\alpha)e_3e_2e_3
-2\gamma e_1e_3=-e_2(-e_2e_3+2\beta)e_3-2\alpha e_1(-e_2e_3+2\beta)e_3-2\gamma e_1e_3
= 1 -2\alpha e_1e_2-2\beta e_2e_3 -(2\gamma+4\alpha\beta)e_1e_3$,
which finally shows that for every element of our basis $x$ the result of the left multiplication  
$e_1e_2e_3\cdot x$ never contains a nonzero $x$-component, so that $Tr(\rho_{reg}(e_1e_2e_3))=0$.

\medskip

\subsection{\bf The orthogonal decomposition $\mathcal H=\RR\cdot 1 \oplus V \oplus 
\widetilde{V}  \oplus \RR\cdot c$.}
We recall that $c=\beta e_1-\gamma e_2 +\alpha e_3-e_1 e_2 e_3$.
From  \ref{subsec-triple-trace} we get $1\perp c$. The orthogonality of 1 to the rest is clear, moreover, for the anticommuting 
pairs of elements of bases of $V$ and $\widetilde{V}$ we get $e_1 \perp \alpha-e_1e_2,\gamma-e_3e_1$ and similarly for $e_2,e_3$. The fact that for non-anticommuting pairs 
like $e_1,\beta-e_2e_3$ their symmetric product $e_1(\beta-e_2e_3)+(\beta-e_2e_3)e_1$ lands in 
$\RR\cdot c$ (see an explicit calculation of that in \ref{subsec-square-map}) implies that even for non-anticommuting pairs we have the orthogonality, 
$e_1\perp \beta-e_2e_3$ and similarly for $e_2,e_3$. 

Next, the orthogonality $c\perp V$
means that we need to check that $Tr(\rho_{reg}(e_j\cdot c))=0$, for example,
$e_1(\beta e_1-\gamma e_2 +\alpha e_3-e_1e_2e_3)=-\beta-\gamma e_1e_2 +\alpha e_1e_3+e_2e_3$
and the corresponding trace of this element is zero. 

The orthogonality $c\perp \widetilde{V}$ means that we need to check
that $Tr(\rho_{reg}(\alpha-e_1e_2)\cdot c)=0$ etc. Here
$(\alpha-e_1e_2)\cdot (\beta e_1-\gamma e_2 +\alpha e_3-e_1e_2e_3)=
\alpha\cdot c-\beta e_1(-e_1e_2+2\alpha)-\gamma e_1 -\alpha e_1e_2e_3 +e_1e_2\cdot e_1e_2e_3
=\alpha \cdot c -\beta e_2-(2\alpha\beta +\gamma)e_1-\alpha e_1e_2e_3 +
e_1(-e_1e_2+2\alpha)e_2e_3=
\alpha \cdot c -\beta e_2-(2\alpha\beta +\gamma)e_1-e_3+\alpha e_1e_2e_3$
and the corresponding trace of the latter element is clearly zero.

\medskip

\subsection{\bf The centrality of $c$} First, in order to see that $c$ indeed belongs to the center $\mathcal Z(\mathcal H)$
it is necessary and sufficient to check that $ce_i=e_ic, i=1,2,3$. Let us check that
$ce_1=e_1c$, the other cases are done similarly.

We have $ce_1=(\beta e_1-\gamma e_2+\alpha e_3 -e_1e_2e_3)e_1=\beta e_1e_1-\gamma (-e_1e_2+
2\alpha)+\alpha (-e_1e_3+2\gamma)-e_1e_2e_3e_1
=e_1(\beta e_1+\gamma e_2-\alpha e_3-e_2e_3e_1)$.
Now $e_2e_3e_1=e_2(-e_1e_3+2\gamma)=-(-e_1e_2+2\alpha)e_3+2\gamma e_2
=e_1e_2e_3+2\gamma e_2-2\alpha e_3$, so that 
$ce_1=e_1(\beta e_1+\gamma e_2-\alpha e_3-(e_1e_2e_3+2\gamma e_2-2\alpha e_3))
=e_1(\beta e_1-\gamma e_2+\alpha e_3 -e_1e_2e_3)=e_1c$.

\medskip

\subsection{\bf The center of $\mathcal H$.}  We know already that, for $c=\beta e_1-\gamma e_2+\alpha e_3 -e_1e_2e_3$ as above, 
we have that $\langle 1,c\rangle \subset \mathcal Z(\mathcal H)$.
Now, for an element $z \in \mathcal H$, $z=x_0\cdot 1+x_1e_1+x_2e_2+x_3e_3+
y_{12}(\alpha-e_1e_2)+y_{23}(\beta-e_2e_3)+y_{31}(\gamma-e_3e_1)+x_4\cdot c$, 
$x_i,y_{ij} \in \RR$, to be in the center means $ze_i=e_iz$ for $i=1,2,3$.
Let us check when there exist such $z$.

As the part $x_0\cdot 1+x_4\cdot c$ is already in $ \mathcal Z(\mathcal H)$, we may assume
that $x_0=x_4=0$. Now the difference
$e_1z-ze_1=x_2(e_1e_2-e_2e_1)+x_3(e_1e_3-e_3e_1)+y_{12}(e_2+e_1e_2e_1)
+y_{23}(-e_1e_2e_3+e_2e_3e_1)+y_{31}(-e_1e_3e_1-e_3)=
 -2x_2(\alpha-e_1e_2)+2x_3(\gamma-e_3e_1)+2y_{12}(e_2+\alpha e_1)
+y_{23}(-e_1e_2e_3+e_2(-e_1e_3+2\gamma))-2y_{31}(e_3+\gamma e_1)=
-2x_2(\alpha-e_1e_2)+2x_3(\gamma-e_3e_1)+2y_{12}(e_2+\alpha e_1)
+y_{23}(-e_1e_2e_3-(-e_1e_2+2\alpha)e_3+2\gamma e_2)-2y_{31}(e_3+\gamma e_1)=
-2x_2(\alpha-e_1e_2)+2x_3(\gamma-e_3e_1)+2y_{12}(e_2+\alpha e_1)
+y_{23}(-e_1e_2e_3+e_1e_2e_3-2\alpha e_3+2\gamma e_2)-2y_{31}(e_3+\gamma e_1)=
-2x_2(\alpha-e_1e_2)+2x_3(\gamma-e_3e_1)+2y_{12}(e_2+\alpha e_1)
+2y_{23}(\gamma e_2-\alpha e_3)-2y_{31}(e_3+\gamma e_1)=
-2x_2(\alpha-e_1e_2)+2x_3(\gamma-e_3e_1)+2(y_{12} \alpha-y_{31}\gamma)e_1  
+2(y_{12}+y_{23}\gamma) e_2 - 2(y_{23}\alpha+y_{31}) e_3
$ equals zero if and only if $x_2=x_3=0$ and $y_{12}\alpha-y_{31}\gamma=0, y_{12}=-y_{23}\gamma$
and $y_{31}=-y_{23}\alpha$. Note that the last two equalities trivially imply the first of the three 
last equalities, so that we see that $z=x_1e_1+y_{23}(-\gamma(\alpha- e_1e_2)+(\beta-e_2e_3)-\alpha(\gamma-e_3e_1))$. Next, commutation with $e_2,e_3$ means that $x_1=0$
and that $z=y_{23}(-\gamma(\alpha- e_1e_2)+(\beta-e_2e_3)-\alpha(\gamma-e_3e_1))
=y_{31}(-\beta(\alpha-e_1e_2)-\alpha(\beta-e_2e_3)+(\gamma-e_3e_1))=
y_{12}((\alpha-e_1e_2)-\gamma(\beta-e_2e_3)-\beta(\gamma-e_3e_1))$,
that is, if $z\neq 0$ then we must have that the following matrix
\[\left(\begin{array}{ccc}
1 & -\gamma & -\beta\\
-\gamma & 1 & -\alpha \\
-\beta & -\alpha & 1
\end{array}\right)
\]
is of rank 1. This is precisely the case when the form $Q$ has signature $(0,1,2)$
($q$ has signature $(1,1,6)$), that is, $|\alpha|=|\beta|=|\gamma|=1$ and $\gamma=-\alpha\beta$.

In this case the center of $\mathcal H$ is spanned by $1,c$ and $z$.

\medskip

\subsection{\bf Equality $c^2=-\det\,Q_{\alpha,\beta,\gamma}$.} Now let us calculate $c^2=(\beta e_1-\gamma e_2+\alpha e_3 -e_1e_2e_3)(\beta e_1-\gamma e_2+\alpha e_3 -e_1e_2e_3)=
-\beta^2-\gamma^2-\alpha^2-\beta\gamma(e_1e_2+e_2e_1)-\alpha\gamma(e_2e_3+e_3e_2)
+\alpha\beta(e_1e_3+e_3e_1)-\beta(e_1\cdot e_1e_2e_3+e_1e_2e_3\cdot e_1)+\gamma(e_2\cdot e_1e_2e_3
+e_1e_2e_3\cdot e_2)-\alpha(e_3\cdot e_1e_2e_3+e_1e_2e_3\cdot e_3)+(e_1e_2e_3)^2=
-\alpha^2-\beta^2-\gamma^2-2\alpha\beta\gamma-\beta(-e_2e_3+(-e_2e_1+2\alpha)e_3e_1)
+\gamma((-e_1e_2+2\alpha)e_2e_3+e_1e_2(-e_2e_3+2\beta))-\alpha((-e_1e_3+2\gamma)e_2e_3-e_1e_2)+(e_1e_2e_3)^2=
-\alpha^2-\beta^2-\gamma^2-2\alpha\beta\gamma-\beta(-e_2e_3-e_2(-e_3e_1+2\gamma)e_1+2\alpha e_3e_1)
+\gamma(e_1e_3+2\alpha e_2e_3+e_1e_3+2\beta e_1e_2)-\alpha(-e_1(-e_2e_3+2\beta)e_3+2\gamma e_2e_3-e_1e_2)+(e_1e_2e_3)^2=
-\alpha^2-\beta^2-\gamma^2-2\alpha\beta\gamma-\beta(-2e_2e_3-2\gamma e_2e_1+2\alpha e_3e_1)
+\gamma(2e_1e_3+2\alpha e_2e_3+2\beta e_1e_2)-\alpha(-2e_1e_2-2\beta e_1e_3+2\gamma e_2e_3)+(e_1e_2e_3)^2=
-\alpha^2-\beta^2-\gamma^2-2\alpha\beta\gamma-\beta(-2e_2e_3-2\gamma (-e_1e_2+2\alpha)+2\alpha e_3e_1)
+\gamma(2(-e_3e_1+2\gamma)+2\alpha e_2e_3+2\beta e_1e_2)-\alpha(-2e_1e_2-2\beta (-e_3e_1+2\gamma)+2\gamma e_2e_3)+(e_1e_2e_3)^2=
-\alpha^2-\beta^2-\gamma^2-2\alpha\beta\gamma-\beta(-2e_2e_3+2\gamma e_1e_2-4\alpha\gamma+2\alpha e_3e_1)
+\gamma (-2e_3e_1+4\gamma+2\alpha e_2e_3+2\beta e_1e_2)-\alpha(-2e_1e_2+2\beta e_3e_1-4\beta\gamma+2\gamma e_2e_3)+(e_1e_2e_3)^2=
-\alpha^2-\beta^2-\gamma^2+6\alpha\beta\gamma+4\gamma^2+2\alpha e_1e_2+
2\beta e_2e_3-2(\gamma+2\alpha\beta) e_3e_1+(e_1e_2e_3)^2=
-\alpha^2-\beta^2-\gamma^2+6\alpha\beta\gamma+4\gamma^2+2\alpha e_1e_2+
2\beta e_2e_3-2(\gamma+2\alpha\beta) e_3e_1+e_1e_2(-e_1e_3+2\gamma)e_2e_3=
-\alpha^2-\beta^2-\gamma^2+6\alpha\beta\gamma+4\gamma^2+2\alpha e_1e_2+
2\beta e_2e_3-2(\gamma+2\alpha\beta) e_3e_1-e_1e_2e_1e_3e_2e_3-2\gamma e_1e_3=
-\alpha^2-\beta^2-\gamma^2+6\alpha\beta\gamma+2\alpha e_1e_2+
2\beta e_2e_3-4\alpha\beta e_3e_1-e_1(-e_1e_2+2\alpha)e_3e_2e_3=
-\alpha^2-\beta^2-\gamma^2+6\alpha\beta\gamma+2\alpha e_1e_2+
2\beta e_2e_3-4\alpha\beta e_3e_1-e_2e_3e_2e_3-2\alpha e_1e_3e_2e_3=
-\alpha^2-\beta^2-\gamma^2+6\alpha\beta\gamma+2\alpha e_1e_2+
2\beta e_2e_3-4\alpha\beta e_3e_1-e_2(-e_2e_3+2\beta)e_3-2\alpha e_1(-e_2e_3+2\beta)e_3=
1-\alpha^2-\beta^2-\gamma^2+6\alpha\beta\gamma
-4\alpha\beta (e_1e_3 +e_3e_1)=1-\alpha^2-\beta^2-\gamma^2-2\alpha\beta\gamma=-\det\,Q_{\alpha,\beta,\gamma}$, as 
was stated.

\medskip

\subsection{\bf Inclusions $cV\subset \widetilde{V},c\widetilde{V}\subset V$.} 

From the above shown, as $ce_1=e_1c=-\beta -\gamma e_1e_2 +\alpha e_1e_3+e_2e_3=
-(\beta-e_2e_3)+\gamma (\alpha-e_1e_2)-\alpha\gamma+\alpha (-e_3e_1+2\gamma)=
-(\beta-e_2e_3)+\alpha (\gamma-e_3e_1)+\gamma (\alpha-e_1e_2) \in \widetilde{V}$.

Analogously one shows that $ce_2,ce_3 \in \widetilde{V}$, that is $cV\subset \widetilde{V}$.

Now let us show that $c\widetilde{V}\subset V$. Again, we are going to give a computation for some base element
of $\widetilde{V}$, leaving to the reader the similar computations with the remaining base elements.
Consider $c(\beta-e_2e_3)=
(\beta e_1-\gamma e_2+\alpha e_3 -e_1e_2e_3)(\beta-e_2e_3)=
\beta^2 e_1 -\beta\gamma e_2+\beta \alpha e_3-\beta e_1e_2e_3
-\beta e_1e_2e_3-\gamma e_3-\alpha e_3e_2e_3+e_1e_2e_3e_2e_3=
\beta^2 e_1 -\beta\gamma e_2+(\alpha\beta-\gamma) e_3-2\beta e_1e_2e_3
-\alpha (-e_2e_3+2\beta)e_3+e_1e_2(-e_2e_3+2\beta)e_3=
(\beta^2 -1)e_1 -(\beta\gamma+\alpha) e_2-(\alpha\beta+\gamma) e_3 \in V$.

\medskip

\subsection{\bf The square map $S\colon v\mapsto v^2$ sends $V \oplus \widetilde{V}$ to $\langle 1,c \rangle \subset \mathcal Z(\mathcal H)$.} 
\label{subsec-square-map}
It is easy to see that each of $V$ and $\widetilde{V}$ is mapped to $\RR\cdot 1$,
moreover as we have that every basis element in $V$ $e_1,e_2,e_3$ anticommutes with the respective two elements of the basis $\alpha-e_1e_2,\beta-e_2e_3,\gamma-e_3e_1$, we see that
in order to prove that $S(V\oplus \widetilde{V}) \subset \langle 1,c\rangle$ we only need to 
consider the symmetric products of the non-anticommuting elements, for example, $e_1(\beta-e_2e_3)+(\beta-e_2e_3)e_1=
2\beta e_1-e_1e_2e_3-e_2e_3e_1=2\beta e_1-e_1e_2e_3-e_2(-e_1e_3+2\gamma)=
2\beta e_1-2\gamma e_2 -e_1e_2e_3+e_2e_1e_3=2\beta e_1-2\gamma e_2 -e_1e_2e_3+(-e_1e_2+2\alpha)e_3=2(\beta e_1-\gamma e_2 +\alpha e_3-e_1e_2e_3)=2c$
and similarly for other respective pairs of basis elements from $V$ and $\widetilde{V}$.

\end{document}